\documentclass[12pt]{amsart}
\usepackage[mathscr]{eucal}
\usepackage{amsmath}
\usepackage{amssymb}
\usepackage{latexsym}
\usepackage[dvips]{graphicx}

\theoremstyle{plain}
	\newtheorem{theorem}{Theorem}[section]
	\newtheorem{lemma}[theorem]{Lemma}
	\newtheorem{corollary}[theorem]{Corollary}
	
	\newtheorem{proposition}[theorem]{Proposition}
	\newtheorem{remark}[theorem]{Remark}
	\newtheorem{conjecture}[theorem]{Conjecture}

\theoremstyle{plain}
	\newtheorem{maintheorem}{Theorem}
	
	\newtheorem{maincorollary}[maintheorem]{Corollary}

\def\R{\mathbb{R}}

\def\bbS{\mathbb{S}}

\def\calB{\mathcal{B}}

\def\calF{\mathcal{F}}

\def\calH{\mathcal{H}}
\def\calL{\mathcal{L}}

\def\calZ{\mathcal{Z}}

\def\Z{\mathbb{Z}}

\def\e{\varepsilon}
\def\tu{\tilde{u}}
\def\bu{\bar{u}}
\def\wtB{\widetilde{B}}
\def\whB{\widehat{B}}
\def\wtR{\widetilde{R}}
\def\tpsi{\tilde{\psi}}

\makeatletter
 \@addtoreset{equation}{section}
\makeatother

\begin{document}


\title[Emden-Fowler equation]{The Emden-Fowler equation\\ on a spherical cap of $\bbS^N$}
\thanks{The second author was supported by JSPS KAKENHI Grant Number 24740100.}

\author{Atsushi Kosaka}
\address{Bukkyo University,
96, Kitahananobo-cho, Murasakino, Kita-ku, Kyoto 603-8301, Japan}
\email{a-kosaka@bukkyo-u.ac.jp}

\author{Yasuhito Miyamoto}
\address{Graduate School of Mathematical Sciences, The University of Tokyo,
3-8-1 Komaba, Meguro-ku, Tokyo 153-8914, Japan}
\email{miyamoto@ms.u-tokyo.ac.jp}

\begin{abstract}
Let $\bbS^N\subset\R^{N+1}$, $N\ge 3$, be the unit sphere, and let $S_{\Theta}\subset\bbS^N$ be a geodesic ball with geodesic radius $\Theta\in(0,\pi)$.
We study the bifurcation diagram $\{(\Theta,\left\|U\right\|_{\infty})\}\subset\R^2$ of the radial solutions of the Emden-Fowler equation on $S_{\Theta}$
\[
\begin{cases}
\Delta_{\bbS^N}U+U^p=0 & \textrm{in}\ S_{\Theta},\\
U=0 & \textrm{on}\ \partial S_{\Theta},\\
U>0 & \textrm{in}\ S_{\Theta},
\end{cases}
\]
where $p>1$.
Among other things, we prove the following: For each $p>p_{\rm S}:=(N-2)/(N+2)$, there exists $\underline{\Theta}\in(0,\pi)$ such that the problem has a radial solution for $\Theta\in(\underline{\Theta},\pi)$ and has no radial solution for $\Theta\in(0,\underline{\Theta})$.
Moreover, this solution is unique in the space of radial functions if $\Theta$ is close to $\pi$.
If $p_{\rm S}<p<p_{\rm JL}$, then there exists $\Theta^*\in(\underline{\Theta},\pi)$ such that the problem has infinitely many radial solutions for $\Theta=\Theta^*$, where
\[
p_{\rm JL}=\begin{cases}
1+\frac{4}{N-4-2\sqrt{N-1}} & \textrm{if}\ N\ge 11,\\
\infty & \textrm{if}\ 2\le N\le 10.
\end{cases}
\]
Asymptotic behaviors of the bifurcation diagram as $p\to\infty$ and $p\downarrow 1$ are also studied.
\end{abstract}

\date{\today}
\maketitle



\section{Introduction and Main results}
Let $\bbS^N\subset\R^{N+1}$, $N\ge 3$, be the unit sphere, and let $S_{\Theta}\subset\bbS^N$ be the geodesic ball centered at the North Pole with geodesic radius $\Theta\in(0,\pi)$.
We call $S_{\Theta}$ the spherical cap.
In this paper we are concerned with the solution of the Emden-Fowler equation on $S_{\Theta}$
\begin{equation}\label{EFS}
\begin{cases}
\Delta_{\bbS^N}U+U^p=0 & \textrm{in}\ \ S_{\Theta},\\
U=0 & \textrm{on}\ \ \partial S_{\Theta},\\
U>0 & \textrm{in}\ \ S_{\Theta},
\end{cases}
\end{equation}
where $\Delta_{\bbS^N}$ denotes the Laplace-Beltrami operator on $\bbS^N$ and $p>1$.
In the Euclidean case it is well known that the qualitative property of the structure of the solutions of the problem
\begin{equation}\label{EFE}
\begin{cases}
\Delta U+U^p=0 & \textrm{in}\ \ B_{\Lambda},\\
U=0 & \textrm{on}\ \ \partial B_{\Lambda},\\
U>0 & \textrm{in}\ \ B_{\Lambda}
\end{cases}
\end{equation}
depends on $p$, and does not depend on $\Lambda$.
Here, $B_{\Lambda}\subset\R^N$ denotes the ball centered at the origin $O$ with radius $\Lambda>0$.
By the symmetry result of Gidas, {\it et al.}\cite{GNN79}, every solution of (\ref{EFE}) is radially symmetric.
The critical Sobolev exponent
\[
p_{\rm S}:=
\begin{cases}
\frac{N+2}{N-2}, & \textrm{if}\ N\ge 3,\\
\infty, & \textrm{if}\ N=1,2
\end{cases}
\]
plays an important role.
It is known that (\ref{EFE}) has a unique solution if $1<p<p_{\rm S}$, and has no solution if $p\ge p_{\rm S}$ (See Poho\'zaev~\cite{P65}).
In the hyperbolic space the moving plane method is applicable and every positive solution of a semilinear elliptic equation with general nonlinearity on a geodesic ball with radius $\Lambda>0$ is radially symmetric. See \cite{KP98,SW12} for this symmetry result.
Bonforte, {\it et al.} \cite{BGGV13} showed, among other things, that in the hyperbolic space the Emden-Fowler equation on the geodesic ball with radius $\Lambda>0$ has a unique positive solution if $1<p<p_{\rm S}$, and has no solution if $p\ge p_{\rm S}$.
Thus, the hyperbolic case is qualitatively the same as the Euclidean case.
In the spherical case Padilla~\cite{P97} and Kumaresan-Prajapat~\cite{KP98} showed that if $S_{\Theta}$ is included in a hemisphere ($0<\Theta<\frac{\pi}{2}$), then every positive solution of a semilinear elliptic equation with general nonlinearity is radially symmetric.
On the other hand, if $S_{\Theta}$ includes a hemisphere ($\frac{\pi}{2}<\Theta<\pi$), then there is a semilinear elliptic equation such that it has a nonradial positive solution.
See \cite{BW07,Mi13} for the existence of nonradial positive solutions.
As far as (\ref{EFS}) is concerned, if $0<\Theta<\pi$ and $1<p\le p_{\rm S}$, then one can easily show that the solution is radial, changing variables and applying the symmetry result of \cite{GNN79} to the equation.
When $\Theta=\frac{\pi}{2}$ and $p>1$, the radial symmetry of a solution of (\ref{EFS}) is guaranteed by \cite[Theorem~1]{SW12}.
The question whether a solution of (\ref{EFS}) is radial in the case where $\frac{\pi}{2}<\Theta<\pi$ and $p>p_{\rm S}$ seems to remain open.
In this paper we restrict ourselves to radially symmetric solutions.

This study is motivated by the result of Bandle-Peletier~\cite{BP99}.
In the case where $N=3$ and $p=p_{\rm S}(=5)$ they showed that (\ref{EFS}) has no solution if $S_{\Theta}$ is included in a hemisphere, and has a radial solution if $S_{\Theta}$ includes a hemisphere.
This indicates that the solution structure depends not only on $p$ but also on the radius $\Theta$.
Actually, we will see in Corollary~\ref{B} below that (\ref{EFS}) has a solution even in the supercritical case $p>p_{\rm S}$ if $\Theta$ is close to $\pi$.
Hence, the solution structure in the spherical case is different from the solution structures in both the Euclidean and hyperbolic cases.
The difference between the Euclidean and spherical cases was also found in the structure of the positive solutions of the Brezis-Nirenberg problem
\[
\begin{cases}
\Delta_{\bbS^3}u+\lambda u+u^5=0 & \textrm{in}\ S_{\Theta}(\subset\bbS^3),\\
u=0 & \textrm{on}\ \partial S_{\Theta}
\end{cases}
\]
which involves the critical Sobolev exponent.
See \cite{BB02,BP06} for details.
It seems that the present paper is the first attempt to study the supercritical Emden-Fowler equation on a spherical cap.
The supercritical Emden-Fowler equation on other manifolds was studied in Berchio, {\it et al.}\cite{BFG14}.

Let us explain the problem in detail.
Let $\theta$ be the geodesic distance from the North Pole of $\bbS^N$.
Let $p>1$ be fixed.
Then the solution $U$ of (\ref{EFS}) depends only on $\theta$.
The problem (\ref{EFS}) can be reduced to the ODE
\begin{equation}
\begin{cases}\label{EFODE}
U''+(N-1)\frac{\cos\theta}{\sin\theta}U'+U^p=0, & 0<\theta<\Theta,\\
U(\Theta)=0,\\
U>0, & 0\le\theta<\Theta.
\end{cases}
\end{equation}
We consider the possibly sign-changing solution of the initial value problem
\begin{equation}\label{IVP}
\begin{cases}
U''+(N-1)\frac{\cos\theta}{\sin\theta}U'+|U|^{p-1}U=0, & 0<\theta<\pi,\\
U(0)=\Gamma>0,\ U'(0)=0.\\
\end{cases}
\end{equation}
In Lemma~\ref{S3L2} we will see that the regular solution $U(\,\cdot\,)$ of (\ref{IVP}) has the first positive zero $\Theta(\Gamma)\in(0,\pi)$.
In Theorem~\ref{A} below we show that $\Theta(\Gamma)$ is a $C^1$-function defined on $0<\Gamma<\infty$.
It is clear that $U(\theta)$ $(0\le\theta\le\Theta(\Gamma))$ is decreasing.
Hence, $\|U\|_{C^0(S_{\Theta(\Gamma)})}=\Gamma$.
The set of all the regular radial solutions of (\ref{EFS}) can be represented by the bifurcation diagram $\{(\Theta(\Gamma),\Gamma)\}\subset\R^2$.
Thus, in this paper we mainly study the graph of the function $\Theta(\Gamma)$.

By $p_{\rm JL}$ we define the Joseph-Lundgren exponent \cite{JL73}, i.e.,
\[
p_{\rm JL}:=
\begin{cases}
1+\frac{4}{N-4-2\sqrt{N-1}}, & \textrm{if}\ N\ge 11,\\
\infty, & \textrm{if}\ 2\le N\le 10.
\end{cases}
\]
\begin{maintheorem}[Supercritical]\label{A}
Suppose that $N\ge 3$ and $p>p_{\rm S}$.
Let $\Theta(\Gamma)$ be the first positive zero of the solution of (\ref{IVP}).
Then the following hold:\\
(i) The function $\Theta(\Gamma)$ is of class $C^1$. For each $\Gamma>0$, $0<\Theta(\Gamma)<\pi$.\\
(ii) $\Theta(\Gamma)\rightarrow\pi$ as $\Gamma\downarrow 0$. If $\Gamma>0$ is small, then $\Theta'(\Gamma)<0$.\\
(iii) $\Theta(\Gamma)\rightarrow \Theta^*$ as $\Gamma\rightarrow\infty$, where $\Theta^*\in(0,\pi)$ is defined in Theorem~\ref{C} below.\\
(iv) If $p_{\rm S}<p<p_{\rm JL}$, then $\Theta(\Gamma)$ oscillates infinitely many times around $\Theta^*$ as $\Gamma\rightarrow\infty$.
\end{maintheorem}
\noindent
See Figure~\ref{FIG} (a) for the bifurcation diagram in the case $p_{\rm S}<p<p_{\rm JL}$.
\begin{figure}[t]
\begin{center}
\includegraphics{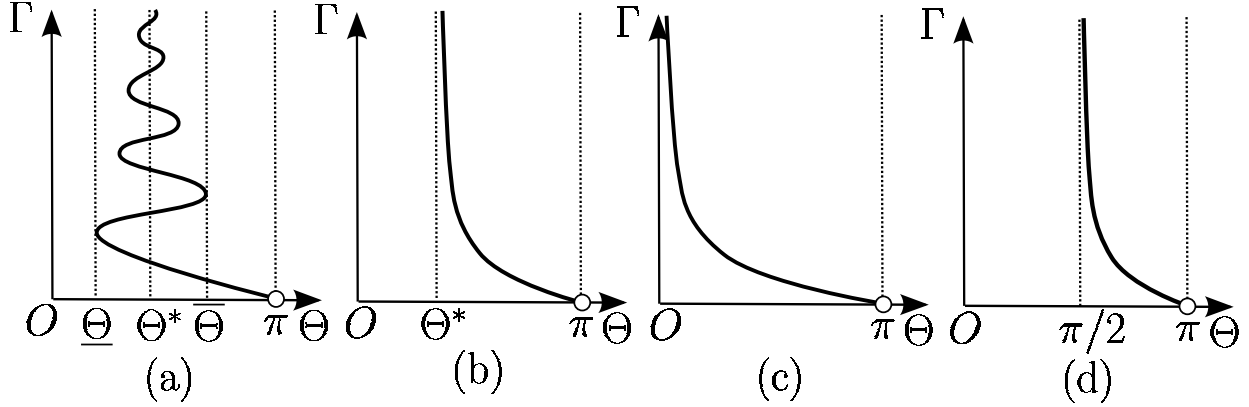}
\label{FIG}
\caption{Schematic bifurcation diagrams: (a) $p_{\rm S}<p<p_{\rm JL}$ (Theorem~\ref{A}), (b) $p\ge p_{\rm JL}$ (Conjecture~\ref{Conj1}), (c) $1<p\le p_{\rm S}$ $(N\ge 4)$ and $1<p<p_{\rm S}$ $(N=3)$ (Proposition~\ref{CriSub}), (d) $p=p_{\rm S}$ and $N=3$ (Proposition~\ref{CriSub}).}
\end{center}
\end{figure}
When $3\le N\le 10$, $p_{\rm JL}=\infty$, and hence, (iv) always holds.
An immediate consequence of Theorem~\ref{A} is the following:
\begin{maincorollary}\label{B}
Suppose that $N\ge 3$ and $p>p_{\rm S}$.
Then the following hold:\\
(i) There exists $\underline{\Theta}>0$ such that (\ref{EFODE}) has no regular solution for $\Theta\in(0,\underline{\Theta})$ and has a regular solution for $\Theta\in(\underline{\Theta},\pi)$.\\
(ii) If $p_{\rm S}<p<p_{\rm JL}$, then (\ref{EFODE}) has a regular solution for $\Theta=\underline{\Theta}$, where $\underline{\Theta}$ is given in (i).\\
(iii) If $p_{\rm S}<p<p_{\rm JL}$, then (\ref{EFODE}) has infinitely many regular solutions for $\Theta=\Theta^*$, where $\Theta^*$ is given in Theorem~\ref{C} below.\\
(iv) There exists $\overline{\Theta}\in (0,\pi)$ such that (\ref{EFODE}) has a unique regular solution for $\Theta\in (\overline{\Theta},\pi)$. This solution is nondegenerate in the space of radial functions.
\end{maincorollary}
\noindent
The problem (\ref{EFODE}) has a singular solution $U^*(\theta)$ such that $U^*(\theta)=O(\theta^{-\frac{2}{p-1}})$ $(\theta\downarrow 0)$.
\begin{maintheorem}\label{C}
Suppose that $N\ge 3$ and $p>p_{\rm S}$.
There exists $\Theta^*\in (0,\pi)$ such that (\ref{EFODE}) has a singular solution $U^*(\theta)$ for $\Theta=\Theta^*$ such that $U^*(\theta)\in C^2(0,\Theta^*]$ and
\begin{equation}\label{CE0}
U^*(\theta)=a\left(\cos\frac{\theta}{2}\right)^{-(N-2)}\left(2\tan\frac{\theta}{2}\right)^{-\mu}(1+o(1))\quad\textrm{as}\quad \theta\downarrow 0,
\end{equation}
where
\begin{equation}\label{Anu}
a:=\left\{\mu(N-2-\mu)\right\}^{\mu/2}\ \ \textrm{and}\ \ \mu:=\frac{2}{p-1}.
\end{equation}
\end{maintheorem}
In the next theorem we obtain the behavior of the curve $\{(\Theta(\Gamma),\Gamma)\}$ for large $p$.
\begin{maintheorem}\label{ThD}
Suppose that $N\ge 3$.
Let $\underline{\Theta}$ be given in Corollary~\ref{B}~(i), and let $\Theta^*$ be given in Theorem~\ref{C}.
Then,
\[
\underline{\Theta}\to\pi\ \textrm{as}\ p\to\infty.
\]
Since $\underline{\Theta}\le\Theta^*$, it holds that $\Theta^*\to\pi$ as $p\to\infty$.
In particular, when $N=3$, $\underline{\Theta}\ge\pi-\arcsin\frac{4}{p-1}$ for $p\ge p_{\rm S}(=5)$.
\end{maintheorem}
In Theorems~\ref{A} and \ref{ThD} detailed properties of $\Theta(\Gamma)$ in the case $p\ge p_{\rm JL}$ are not clarified.
\begin{conjecture}\label{Conj1}
Suppose that $N\ge 11$.
If $p\ge p_{\rm JL}$, then $\Theta(\Gamma)$ is strictly decreasing and (\ref{EFODE}) has no regular solution for $\Theta\in(0,\Theta^*]$.
\end{conjecture}
\noindent
Figure~\ref{FIG} (b) shows a conjectured bifurcation diagram in the case $p\ge p_{\rm JL}$.

Next, we consider the critical case $p=p_{\rm S}$ and subcritical case $1<p<p_{\rm S}$.
The following proposition follows from combining known results~\cite{BBF98,BP99,SW13} and our results.
\begin{proposition}[Critical/Subcritical]\label{CriSub}
Suppose that $N\ge 3$ and $1<p\le p_{\rm S}$.
Let $\Theta(\Gamma)$ be the first positive zero of the solution of (\ref{IVP}).\\
(i) The function $\Theta(\Gamma)$ is of class $C^1$. For each $\Gamma>0$, $0<\Theta(\Gamma)<\pi$.\\
(ii) $\Theta(\Gamma)\rightarrow\pi$ as $\Gamma\downarrow 0$.\\
(iii) $\Theta(\Gamma)$ is strictly decreasing.\\
(iv) If $N\ge 4$, then $\Theta(\Gamma)\to 0$ as $\Gamma\to\infty$.\\
(v) If $N=3$ and $p=p_{\rm S}(=5)$, then $\Theta(\Gamma)\to\frac{\pi}{2}$ as $\Gamma\to\infty$. On the other hand, if $N=3$ and $1<p<p_{\rm S}$, then $\Theta(\Gamma)\to 0$ as $\Gamma\to\infty$.
In particular, if $N=3$ and $p=p_{\rm S}(=5)$, (\ref{EFODE}) has no regular solution for $\Theta\in(0,\frac{\pi}{2}]$.
\end{proposition}
\noindent
See Figure~\ref{FIG} (c) and (d).

When $1<p<p_{\rm S}$, for each fixed $\Theta_0\in(0,\pi)$, there is a unique $\Gamma_0>0$ depending on $p$ such that $\Theta(\Gamma_0)=\Theta_0$.
Therefore, we write $\Gamma_0$ by $\Gamma(p)$.
The asymptotic shape of the branch as $p\downarrow 1$ is as follows:
\begin{maintheorem}\label{ThE}
Suppose that $N\ge 3$.
There exists $\Theta^{\dagger}\in (0,\pi)$ such that the following statements hold:\\
(i) If $0<\Theta<\Theta^{\dagger}$, then $\Gamma(p)\to\infty$ as $p\downarrow 1$.\\
(ii) If $\Theta=\Theta^{\dagger}$, then $\Gamma(p)\to\Gamma^{\dagger}$ as $p\downarrow 1$ with some constant $\Gamma^{\dagger}>0$.\\
(iii) If $\Theta^{\dagger}<\Theta<\pi$, then $\Gamma(p)\to 0$ as $p\downarrow 1$.
\end{maintheorem}
Since the solution structure changes at $p=p_{\rm S}$, it is natural to study the case where $p\downarrow p_{\rm S}$.
We are led to the following:
\begin{conjecture}
Let $\overline{\Theta}$ be given in Corollary~\ref{B} (iv), and let $\Theta^*$ be given in Theorem~\ref{C}.
If $N\ge 4$, then $\overline{\Theta}\to 0$ ($p\downarrow p_{\rm S}$) and $\Theta^*\to 0$ ($p\downarrow p_{\rm S}$).
If $N=3$, then $\overline{\Theta}\to\frac{\pi}{2}$ ($p\downarrow p_{\rm S}$) and $\Theta^*\to\frac{\pi}{2}$ ($p\downarrow p_{\rm S}$).
\end{conjecture}

Let us explain technical details.
Using the stereographic projection $v(r):=U(\theta)$ and $r:=\tan\frac{\theta}{2}$, we have
\[
v''+\frac{N-1}{r}v'-(N-2)rA(r)v'+A(r)^2v^p=0,
\]
where
\[
A(r):=\frac{2}{1+r^2}.
\]
We let $u(r):=A(r)^{\frac{N-2}{2}}v(r)$.
Then, we have the semilinear elliptic problem
\begin{equation}\label{D}
\begin{cases}
u''+\frac{N-1}{r}u'+\frac{N(N-2)}{4}A(r)^2u+\frac{1}{A(r)^q}u^{p}=0, & 0<r<R,\\
u(R)=0, & \\
u>0, & 0\le r<R,
\end{cases}
\end{equation}
where
\[
R:=\tan\frac{\Theta}{2}\ \ \textrm{and}\ \ 
q:=\frac{N-2}{2}(p-p_{\rm S}).
\]
Note that if $R=1$, then $S_{\Theta}$ is a hemisphere ($\Theta=\frac{\pi}{2}$).
The problem (\ref{IVP}) is equivalent to the problem
\begin{equation}\label{DD}
\begin{cases}
u''+\frac{N-1}{r}u'+\frac{N(N-2)}{4}A(r)^2u+\frac{1}{A(r)^q}|u|^{p-1}u=0, & 0<r<\infty,\\
u(0)=\gamma>0,\ u'(0)=0,
\end{cases}
\end{equation}
where $\gamma:=2^{\frac{N-2}{2}}\Gamma$.
By $R(\gamma)$ we denote the first positive zero of the solution $u(\,\cdot\,,\gamma)$ of (\ref{DD}), i.e., $R(\gamma)=\tan\frac{\Theta(\Gamma)}{2}$.
In this paper we mainly consider (\ref{DD}).

The existence of infinitely many turning points for semilinear elliptic equations on a Euclidean ball was proved by the several authors.
In \cite{DF07,GW11} the Brezis-Nirenberg problem including a supercritical exponent was studied.
Dolbeault-Flores~\cite{DF07} used the geometric theory of dynamical systems.
Guo-Wei~\cite{GW11} used the Morse indices of solutions, using the intersection number between the regular and singular solutions.
See \cite{KW18,Mi14a,Mi14b,Mi18} for other results.
In \cite{D00,D08a,D08b,D13} Dancer studied infinitely many turning points of supercritical semilinear Dirichlet problems on a rather general domain, using the analytic property.
We show that (\ref{EFS}) has a singular solution $U^*$.
Using the intersection number of the singular solution $U^*(R)$ and a regular solution $U(R,\Gamma)$ of (\ref{IVP}) in the interval $I(\gamma)$
\[
\calZ_{I(\gamma)}[U^*(\,\cdot\,)-U(\,\cdot\,,\Gamma)],
\]
we prove the existence of infinitely many turning points as $\Gamma\to\infty$.
Here, $I(\gamma):=(0,\min\{R(\gamma),R^*\})$, $R(\gamma)$ and $R^*$ are the first positive zeros of $U$ and $U^*$, repsectively.


This paper consists of eight sections.
In Section~2 we recall known results about the Emden-Fowler equation on $\R^N$.
In Section~3 we prove Theorem~\ref{A} (i).
In Section~4 we construct the singular solution (Theorem~\ref{C}).
In Sections 5, 6, and 7 we prove Theorem~\ref{A} (iii), (ii), and (iv), respectively.
The proof of Corollary~\ref{B} is in Section~7.
In Section~8 we prove Theorems~\ref{ThD} and \ref{ThE}.
Proposition~\ref{CriSub} is also proved in Section~8.

\section{Known results}
We recall known results about solutions of the equation
\[
\bar{u}''+\frac{N-1}{\rho}\bar{u}'+\bar{u}^p=0,\quad 0<\rho<\infty.
\]
See \cite{JL73,W93} for details.
This problem has the singular solution
\begin{equation}\label{S2E0}
\bar{u}^*(\rho):=a\rho^{-\mu},
\end{equation}
where $a$ and $\mu$ are defined by (\ref{Anu}).
Let $\bar{u}(\rho,\bar{\gamma})$ be the solution of
\begin{equation}\label{S2E1}
\begin{cases}
\bar{u}''+\frac{N-1}{\rho}\bar{u}'+\bar{u}^p=0, & 0<\rho<\infty,\\
\bar{u}(0)=\bar{\gamma}>0,\ \bar{u}'(0)=0.\\
\end{cases}
\end{equation}
We use Emden's transformation
\[
\bar{y}(t):=\frac{\bar{u}(\rho,\bar{\gamma})}{\bar{u}^*(\rho)}\ \ \textrm{and}\ \ t:=\frac{1}{m}\log\rho,
\]
where
\begin{equation}\label{m}
m:=a^{-\frac{p-1}{2}}.
\end{equation}
Then $\bar{y}(t)$ satisfies
\begin{equation}\label{S2E2}
\begin{cases}
\bar{y}''+\alpha\bar{y}-\bar{y}+\bar{y}^p=0,& -\infty<t<\infty,\\
ae^{-m\mu t}\bar{y}(t)\rightarrow\bar{\gamma}\ \textrm{as}\ t\rightarrow-\infty,\\
e^{-mt}(e^{-m\mu t}\bar{y}(t))'\rightarrow 0\ \textrm{as}\ t\rightarrow-\infty,
\end{cases}
\end{equation}
where
\begin{equation}\label{alpha}
\alpha:=m(N-2-2\mu).
\end{equation}
Let $\bar{z}(t):=\bar{y}'(t)$.
Then, $(\bar{y},\bar{z})$ satisfies
\begin{equation}\label{S2E3}
\begin{cases}
\bar{y}'=\bar{z}\\
\bar{z}'=-\alpha\bar{z}+\bar{y}-\bar{y}^p.
\end{cases}
\end{equation}
We study the orbit $(\bar{y}(t),\bar{z}(t))$.
Let
\[
J(\bar{y},\bar{z}):=\frac{\bar{z}^2}{2}-\frac{\bar{y}^2}{2}+\frac{\bar{y}^{p+1}}{p+1}.
\]
By direct calculation we have
\[
\frac{d}{dt}J(\bar{y}(t),\bar{z}(t))=-\alpha\bar{z}(t)^2.
\]
If $p>p_{\rm S}$, then $\alpha>0$, and hence, $\frac{d}{dt}J(\bar{y}(t),\bar{z}(t))\le 0$.
Then, $J$ is a Lyapunov function of (\ref{S2E3}).
We see by the initial condition in (\ref{S2E2}) that $(\bar{y}(-\infty),\bar{z}(-\infty))=(0,0)$.
Therefore, $J(\bar{y}(t),\bar{y}(t))\le 0$ for all $t\in\R$.

The system (\ref{S2E3}) has the unique equilibrium $(1,0)$ in the bounded set $\{(\bar{y},\bar{z})\in\R^2;\ J(\bar{y},\bar{z})<0,\ \bar{y}>0\}$.
It follows from the Poincar\'e-Bendixson theorem that $(\bar{y}(t),\bar{z}(t))\rightarrow (1,0)$ as $t\rightarrow\infty$.
Next, we study the behavior of $(\bar{y}(t),\bar{z}(t))$ near $(1,0)$.
The two eigenvalues of the linearization at $(1,0)$ are given by $\lambda^2+\alpha\lambda+p-1=0$.
Therefore, $(1,0)$ is a stable spiral if $\alpha^2-4(p-1)<0$.
This inequality is equivalent to $(N-2-2\mu)^2-8(N-2-\mu)<0$.
Solving this inequality for $\mu$, we have
\begin{equation}\label{S2E4}
\frac{N-4-2\sqrt{N-1}}{2}<\mu<\frac{N-4+2\sqrt{N-1}}{2}.
\end{equation}
Since $1+4/(N-4+2\sqrt{N-1})<p_{\rm S}<p$, we see that $\mu<(N-4+2\sqrt{N-1})/2$.
If $N\le 10$, then $(N-4-2\sqrt{N-1})/2\le 0<\mu$, and hence (\ref{S2E4}) holds.
In the case $N\ge 11$, (\ref{S2E4}) holds if
\begin{equation}\label{S2E5}
p<1+\frac{4}{N-4-2\sqrt{N-1}}(=p_{\rm JL}).
\end{equation}
We have seen the following: The orbit $(\bar{y}(t),\bar{z}(t))$ starts from $(0,0)$ at $t=-\infty$ and converges to $(1,0)$ as $t\rightarrow\infty$.
Moreover, if (\ref{S2E5}) holds, then $(\bar{y}(t),\bar{z}(t))$ rotates clockwise around $(1,0)$.
Therefore, there is $\{t_j\}_{j=1}^{\infty}$ $(t_1<t_2<\cdots\rightarrow\infty)$ such that $z(t_j)=0$ $(j\in\{1,2,\ldots\})$ and
\[
y(t_2)<y(t_4)<\cdots<y(t_{2j})<\cdots<1<\cdots<y(t_{2j-1})<\cdots<y(t_3)<y(t_1).
\]
This means that $y(t)$ oscillates around $1$ infinitely many times.
Since $\bar{y}(t)=\frac{\bar{u}(\rho,\bar{\gamma})}{\bar{u}^*(\rho)}$, the intersection number between $\bar{u}(\rho,\bar{\gamma})$ and $\bar{u}^*(\rho)$, which we denote by $\calZ_{(0,\infty)}[\bar{u}(\,\cdot\,,\bar{\gamma})-\bar{u}^*(\,\cdot\,)]$, is $\infty$.
\begin{proposition}\label{S2P1}
(i) Let $\bar{u}(\rho,\bar{\gamma})$ be the solution of (\ref{S2E1}).
If $p_{\rm S}<p<p_{\rm JL}$, then $\calZ_{(0,\infty)}[\bar{u}(\,\cdot\,,\bar{\gamma})-\bar{u}^*(\,\cdot\,)]=\infty$.\\
(ii) Let $(\bar{y}(t),\bar{z}(t))$ be the solution of (\ref{S2E2}).
If $p>p_{\rm S}$, then, for each $\bar{\gamma}>0$, $(\bar{y}(t),\bar{z}(t))$ converges to $(1,0)$ as $t\rightarrow\infty$.
\end{proposition}

\section{Parameterization results}
The aim of this section is to show that the regular solutions of (\ref{D}) can be parameterized by $\gamma$.
Parametrization results for Euclidean cases were obtained by several authors.
See \cite{K97,Mi14a} for example.
The proof is similar.
However, we give the proof for readers' convenience.
\begin{lemma}\label{S3L1}
Suppose that $p>1$.
Let $(R_0,u_0(r))$ be a solution of (\ref{DD}) with $\gamma=\gamma_0$.
Then, there is a $C^1$-mapping $\gamma\mapsto (R(\gamma),u(r,\gamma))$ such that all solutions of (\ref{DD}) near $(R_0,u_0(r))$ can be described as $\{(R(\gamma),u(r,\gamma))\}_{|\gamma-\gamma_0|<\varepsilon}$ $(u(0,\gamma)=\gamma)$ and that $(R(\gamma_0),u(r,\gamma_0))=(R_0,u_0(r))$.
\end{lemma}
\begin{proof}
Since $u(r,\gamma)$ is a solution of (\ref{DD}), $u(r,\gamma)$ is a $C^1$-function of $r$ and $\gamma$.
Since $u$ satisfies the equation in (\ref{DD}), $u_r(R_0,\gamma_0)\neq 0$, otherwise $u(r,\gamma_0)\equiv 0$ $(0<r<R)$ by the uniqueness of the solution of the ODE.
Since $u(R_0,\gamma_0)=0$, we can apply the implicit function theorem to $u(r,\gamma)=0$.
Then, there is a $C^1$-function $R=R(\gamma)$ defined on $|\gamma-\gamma_0|<\e$ such that $u(R(\gamma),\gamma)=0$ and $R(\gamma_0)=R_0$.
Because of the continuity of $u(r,\gamma)$, $u(r,\gamma)>0$ in $\{(r,\gamma);\ 0<r<R(\gamma),\ |\gamma-\gamma_0|<\e\}$.
Thus, $(R(\gamma),u(r,\gamma))$ is a solution of (\ref{DD}).
The implicit function theorem also says that all solutions of (\ref{DD}) near $(R_0,u_0(r))$ are $\{(R(\gamma),u(r,\gamma))\}_{|\gamma-\gamma_0|<\e}$ and that the mapping $\gamma\mapsto (R(\gamma),u(r,\gamma))$ is of class $C^1$.
The proof is complete.
\end{proof}

\begin{lemma}\label{S3L2}
Suppose that $p>1$.
Let $U(\theta,\Gamma)$ be the solution of (\ref{IVP}).
Then $U(\,\cdot\,,\Gamma)$ has the first positive zero $\Theta(\Gamma)\in (0,\pi)$.
\end{lemma}
\begin{proof}
Let $U$ be the solution of (\ref{IVP}).
By the equation in (\ref{IVP}) we have
\begin{equation}\label{S3L2E1}
(U'\sin^{N-1}\theta)'+|U|^{p-1}U\sin^{N-1}\theta=0.
\end{equation}
Integrating (\ref{S3L2E1}) over $[0,\theta]$, we have
\begin{equation}\label{S3L2E1+}
U'(\theta)=-\frac{1}{\sin^{N-1}\theta}\int_0^{\theta}|U(\varphi)|^{p-1}U(\varphi)\sin^{N-1}\varphi d\varphi.
\end{equation}
Thus, 
\begin{equation}\label{S3L2E2}
\textrm{if $U(\theta)>0$ for $\theta\in[0,\theta_0)$, then $U'(\theta)<0$ for $\theta\in(0,\theta_0]$}.
\end{equation}
By contradiction we prove the statement of the lemma.
Suppose the contrary, i.e., $U(\theta)>0$ for $\theta\in[0,\pi)$.
By (\ref{S3L2E2}) we see that $U'(\theta)<0$ for $\theta\in(0,\pi)$.

Let $\theta_1$ and $\theta_2$ be such that $0<\theta_1<\theta_2<\pi$.
We let $\theta>\theta_2$.
Integrating (\ref{S3L2E1}) over $[\theta_1,\theta]$, we have
\[
U'(\theta)=-\frac{C(\theta)}{\sin^{N-1}\theta},
\]
where
\[
C(\theta):=|U'(\theta_1)|\sin^{N-1}\theta_1+\int_{\theta_1}^{\theta}|U(\varphi)|^{p-1}U(\varphi)\sin^{N-1}\varphi d\varphi.
\]
We have
\begin{align*}
C(\theta_2)&=|U'(\theta_1)|\sin^{N-1}\theta_1+\int_{\theta_1}^{\theta_2}|U(\varphi)|^{p-1}U(\varphi)\sin^{N-1}\varphi d\varphi\\
&\ge|U'(\theta_1)|\sin^{N-1}\theta_1+U(\theta_2)^p\int_{\theta_1}^{\theta_2}\sin^{N-1}\varphi d\varphi\\
&>0.
\end{align*}
Since $\theta_2<\theta$, $C(\theta_2)<C(\theta)$.
Therefore,
\begin{equation}\label{S3L2E3}
U'(\theta)<-\frac{C(\theta_2)}{\sin^{N-1}\theta}\ \ \textrm{for}\ \ \theta>\theta_2.
\end{equation}
Integrating (\ref{S3L2E3}) over $[\theta_2,\theta]$, we have
\[
U(\theta)\le U(\theta_2)-C(\theta_2)\int_{\theta_2}^{\theta}\frac{d\varphi}{\sin^{N-1}\varphi}.
\]
Hence, $U(\theta)\to -\infty$ as $\theta\uparrow\pi$.
This contradicts the assumption.
Thus, there exists the first positive zero $\Theta(\Gamma)\in(0,\pi)$.
\end{proof}

As we see in the following lemma, the solution set of (\ref{D}) is a curve and it can be parametrized by $\gamma$.
\begin{lemma}\label{S3L3}
Suppose that $p>1$.
There is a $C^1$-mapping $\gamma\mapsto (R(\gamma),u(r,\gamma))$ defined on $(0,\infty)$ such that all regular solutions of (\ref{D}) can be described as $(R(\gamma),u(r,\gamma))$.
Specifically, for each $\gamma>0$, $R(\gamma)$ is defined and $0<R(\gamma)<\infty$.
\end{lemma}
\begin{proof}
Let $u(r,\gamma)$ be the solution of (\ref{DD}).
Because of Lemma~\ref{S3L2}, the solution $u(\,\cdot\,,\gamma)$ of (\ref{DD}) also has the first positive zero $R(\gamma)\in (0,\infty)$.

The first positive zero $R(\gamma)$ is defined for every $\gamma>0$, and $0<R(\gamma)<\infty$ for $\gamma>0$.
By Lemma~\ref{S3L1} we see that $R(\gamma)$ is of class $C^1$.
It is clear that $\{(R(\gamma),u(r,\gamma))\}_{\gamma>0}$ is the set of all regular solutions of (\ref{D}).
The proof is complete.
\end{proof}

\section{Singular solution}
In this section we show that (\ref{D}) has a singular solution $(R^*,u^*(r))$.
Let $u(r)$ be a solution of (\ref{D}).
We use the change of variables
\begin{equation}\label{y}
y(t):=2^{-\frac{q}{p-1}}\frac{u(r)}{\bar{u}^*(r)}\ \ \textrm{and}\ \ t:=\frac{1}{m}\log r.
\end{equation}
Here, $\bar{u}^*(r)$ is defined by (\ref{S2E0}), $m$ is defined by (\ref{m}).
Then $y$ satisfies
\begin{equation}\label{S4E1}
y''+\alpha y'-y+y^p+B_0(t)y^p+B_1(t)y=0,
\end{equation}
where $\alpha$ is defined by (\ref{alpha}),
\begin{equation}\label{S4E1+}
B_0(t):=\left(1+e^{2mt}\right)^q-1,\ \ \textrm{and}\ \ B_1(t):=\frac{N(N-2)e^{2mt}}{(1+e^{2mt})^2}.
\end{equation}
Note that $B_0(t)>0$ and $B_1(t)>0$.

We construct the singular solution near $t=-\infty$.
\begin{lemma}\label{S4L1}
Suppose that $p>p_{\rm S}$.
Assume that the problem
\begin{equation}\label{S4L1E1}
\begin{cases}
y''+\alpha y'-y+y^p+B_0(t)y^p+B_1(t)y=0,\\
y(t)\rightarrow 1\ \textrm{as}\ t\rightarrow -\infty
\end{cases}
\end{equation}
has a solution $y^*(t)$ near $t=-\infty$.
Then, $y^*(t)$ satisfies
\begin{equation}\label{S4L1E1+}
y^*(t)=1+O(e^{2mt})\ \textrm{as}\ t\rightarrow -\infty.
\end{equation}
\end{lemma}
\begin{proof}
Let $\tau:=-t$ and $\eta(\tau):=y(t)-1$.
Then $\eta(\tau)$ satisfies
\begin{equation}\label{S4L1E2}
\begin{cases}
\eta''-\alpha\eta'+(p-1)\eta=g(\tau), & \tau_0<\tau<\infty,\\
\eta(\tau)\rightarrow 0\ \textrm{as}\ \tau\rightarrow\infty,
\end{cases}
\end{equation}
where $\tau_0$ is large,
\begin{equation}\label{S4L1E3}
g(\tau):=-B_0(-\tau)(\eta+1)^p-B_1(-\tau)(\eta+1)-\varphi(\eta),
\end{equation}
\[
\varphi(\eta):=(1+\eta)^p-1-p\eta.
\]
There are three cases:
\begin{equation}\label{S4L1E31}
\textrm{(1)}\ p-1>\left(\frac{\alpha}{2}\right)^2,\quad
\textrm{(2)}\ p-1<\left(\frac{\alpha}{2}\right)^2,\quad
\textrm{(3)}\ p-1=\left(\frac{\alpha}{2}\right)^2.
\end{equation}
We consider only the case (1).
The other cases can be similarly treated.
Because the linearly independent solutions of the homogeneous equation associated with the equation of (\ref{S4L1E2}) becomes unbounded as $\tau\rightarrow\infty$, we have
\[
\eta(\tau)=\frac{e^{\frac{\alpha\tau}{2}}}{\beta}\int_{\tau}^{\infty}e^{-\frac{\alpha}{2}\sigma}\sin (\beta(\sigma-\tau))g(\sigma)d\sigma,
\]
where $\beta:=\sqrt{(p-1)-\left(\frac{\alpha}{2}\right)^2}$.
If $|\eta|$ is small, then there are a small $\e>0$ and $\tau_{\e}$ such that
\begin{equation}\label{S4L1E4}
|\varphi(\tau)|\le |(1+\eta)^p-1-p\eta|\le\e|\eta|\ (\tau>\tau_{\e})
\end{equation}
By (\ref{S4L1E3}) and (\ref{S4L1E4}) we have
\[
|g(\tau)|\le C_0e^{-2m\tau}+\e|\eta(\tau)|\ (\tau>\tau_{\e}).
\]
Using the same method as in the proof of Merle-Peletier~\cite[Lemma~3.1]{MP91}, we have $\eta(\tau)=O(e^{-2m\tau})$ as $\tau\to\infty$.
Therefore, (\ref{S4L1E1+}) holds.
See \cite[Lemma~6.3]{Mi14a} for details.
\end{proof}

\begin{lemma}\label{S4L2}
Suppose that $p>p_{\rm S}$.
The problem (\ref{S4L1E1}) has a unique solution near $t=-\infty$.
\end{lemma}
\begin{proof}
There are three cases (\ref{S4L1E31}) as in the proof of Lemma~\ref{S4L1}.
We consider only the case (1).
We transform (\ref{S4E1}) to the integral equation
\begin{equation}\label{S4L2E1}
\eta(\tau)=\calF(\eta)(\tau).
\end{equation}
In the case (1) $\calF$ becomes
\[
\calF(\eta)(\tau)=\frac{e^{\frac{\alpha\tau}{2}}}{\beta}\int_{\tau}^{\infty}e^{-\frac{\alpha}{2}\sigma}\sin(\beta(\sigma-\tau))g(\sigma)d\sigma.
\]

By $\|\,\cdot\,\|$ we denote $\left\|\,\cdot\,\right\|_{C^0[\tau_0,\infty)}$.
We set $X:=\{\eta(\tau)\in C^0[\tau_0,\infty);\ \|\eta(\tau)\|<\infty\}$ and
$\calB:=\{\eta(\tau)\in X;\ \|\eta\|<\delta\}$.
If $\delta>0$ is small, then we can show that $\calF(\calB)\subset\calB$ and $\calF$ is a contraction mapping on $\calB$, using Lemma~\ref{S4L1}.
By the contraction mapping theorem we see that (\ref{S4L2E1}) has a unique solution in $\calB$.
We omit the detail.
\end{proof}
Let $y^*(t)$ be the solution of (\ref{S4L1E1}) obtained in Lemma~\ref{S4L2}.
We define
\begin{equation}\label{S4E2}
u^*(r)=2^{\frac{q}{p-1}}ar^{-\mu}y^*(\frac{1}{m}\log r).
\end{equation}
\begin{corollary}
Suppose that $p>p_{\rm S}$.
Let $u^*(r)$ be defined by (\ref{S4E2}). Then
\begin{equation}\label{S4C2E1}
u^*(r)=2^{\frac{q}{p-1}}ar^{-\mu}(1+o(1))\ \textrm{as}\ r\downarrow 0.
\end{equation}
\begin{equation}\label{S4C2E2}
(u^*)'(r)=-2^{\frac{q}{p-1}}\mu ar^{-\mu-1}(1+o(1))\ \textrm{as}\ r\downarrow 0.
\end{equation}
\end{corollary}
\begin{proof}
By (\ref{S4E2}) and Lemmas~\ref{S4L1} and \ref{S4L2} we obtain (\ref{S4C2E1}).
Differentiating (\ref{S4L2E1}) in $\tau$, we have
\[
\eta'(\tau)=\frac{\alpha}{2\beta}e^{\frac{\alpha\tau}{2}}\int_{\tau}^{\infty}e^{-\frac{\alpha}{2}\sigma}\sin(\beta(\sigma-\tau))g(\sigma)d\sigma
-e^{\frac{\alpha\tau}{2}}\int_{\tau}^{\infty}e^{-\frac{\alpha}{2}\sigma}\cos(\beta(\sigma-\tau))g(\sigma)d\sigma.
\]
We have that $\eta'(\tau)=O(e^{-2m\tau})$, and hence, 
\begin{equation}\label{S4C2E3-}
(y^*)'(r)=O(e^{2mt}).
\end{equation}
Differentiating (\ref{S4E2}) in $r$, we have
\begin{equation}\label{S4C2E3}
(u^*)'(r)=-2^{\frac{q}{p-1}}\mu ar^{-\mu-1}y^*(\frac{1}{m}\log r)+2^{\frac{q}{p-1}}ar^{-\mu}(y^*)'(\frac{1}{m}\log r)\frac{1}{r}.
\end{equation}
Substituting (\ref{S4C2E3-}) and (\ref{S4L1E1+}) into (\ref{S4C2E3}), we have (\ref{S4C2E2}).
\end{proof}

\begin{corollary}\label{S4C4}
Suppose that $p>p_{\rm S}$.
Let $U^*(\theta):=A(r)^{-\frac{N-2}{2}}u^*(r)$ and $r:=\tan\frac{\theta}{2}$.
Then, (\ref{CE0}) and the following hold:
\begin{equation}\label{S4C4E0}
(U^*)'(\theta)=a\left(\cos\frac{\theta}{2}\right)^{-N}\left(2\tan\frac{\theta}{2}\right)^{-\mu-1}\left(-\mu+(N-2)\left(\sin\frac{\theta}{2}\right)^2+o(1)\right)\ \ \textrm{as}\ \ \theta\downarrow 0.
\end{equation}
\end{corollary}
\begin{proof}
By direct calculation we have (\ref{CE0}).
We have
\begin{align}
\frac{d}{d\theta}U^*(\theta)&=\frac{1}{A(r)}\frac{d}{dr}\left(A(r)^{-\frac{N-2}{2}}u^*(r)\right)\nonumber\\
&=-\frac{N-2}{2}A(r)^{\frac{N-2}{2}}A'(r)u^*(r)+A(r)^{-\frac{N}{2}}(u^*)'(r).\label{S4C4E1}
\end{align}
Substituting (\ref{S4C2E1}) and (\ref{S4C2E2}) into (\ref{S4C4E1}), we obtain (\ref{S4C4E0}).
\end{proof}
Since $u^*(r)$ satisfies the equation in (\ref{D}), $U^*(\theta)$ satisfies the equation in (\ref{EFODE}).
Then the domain of $U^*(\theta)$ can be extended.
In the following lemma we show that $U^*(\theta)$ has the first positive zero, and hence, $U^*(\theta)$ is a singular solution of (\ref{EFODE}).
\begin{lemma}\label{S4L3}
Suppose that $p>p_{\rm S}$.
Let $U^*(\theta):=A(r)^{-\frac{N-2}{2}}u^*(r)$ and $r:=\tan\frac{\theta}{2}$.
Then $U^*(\theta)$ has the first positive zero $\Theta^*\in(0,\pi)$.
Hence, $(\Theta^*,U^*(\theta))$ is the singular solution of (\ref{EFODE}).
\end{lemma}
\begin{proof}
First, we prove
\begin{equation}\label{S4L3E1}
(U^*)'(\theta)\sin^{N-1}\theta\to 0\ \ \textrm{as}\ \ \theta\downarrow 0.
\end{equation}
In fact, $(U^*)'(\theta)\sin^{N-1}\theta=O(\theta^{-\mu-1+N-1})$ and $-\mu-1+N-1=N-2-\frac{2}{p-1}>0$.
Hence, (\ref{S4L3E1}) holds.
Integrating (\ref{S3L2E1}) over $(0,\theta]$, we have (\ref{S3L2E1+}).
Hence, (\ref{S3L2E2}) holds.
The rest of the proof is the same as the proof of Lemma~\ref{S3L2}.
\end{proof}

\begin{proof}[Proof of Theorem~\ref{C}]
The singular solution $(\Theta^*,U^*(\theta))$ is established in Lemma~\ref{S4L3}, and (\ref{CE0}) is obtained in Corollary~\ref{S4C4}.
\end{proof}

\begin{remark}\label{rem1}
Let $(\Theta^*,U^*(\theta))$ be the singular solution of (\ref{EFODE}).
Let $u^*(r):=U^*(\theta)A(\tan\frac{\theta}{2})^{\frac{N-2}{2}}$, $r:=\tan\frac{\theta}{2}$, and $R^*:=\tan\frac{\Theta^*}{2}$.
Then $(R^*,u^*(r))$ is the singular solution of (\ref{D}).
\end{remark}

\section{Convergence to the singular solution as $\gamma\to\infty$}
Let $u(r,\gamma)$ be the solution of (\ref{DD}), and let $R(\gamma)$ be the first positive zero of $u(\,\cdot\,,\gamma)$.
Let $(R^*,u^*(r))$ be the singular solution of (\ref{D}) given in Remark~\ref{rem1}.
Our goal in this section is to prove the following:
\begin{lemma}\label{S5L2}
Suppose that $p>p_{\rm S}$.
Let $(R^*,u^*(r))$ be the singular solution given by Lemma~\ref{S4L3}.
As $\gamma\rightarrow\infty$,
\[
R(\gamma)\rightarrow R^*\ \ \textrm{and}\ \ u(r,\gamma)\rightarrow u^*(r)\ \textrm{in}\ C^2_{loc}(0,R^*].
\]
\end{lemma}
We postpone the proof of Lemma~\ref{S5L2}.
Let $y(t)$ be defined as (\ref{y}).
Then (\ref{DD}) is equivalent to the problem
\begin{equation}\label{S5E1}
\begin{cases}
y''+\alpha y'-y+y^p+B_0(t)y^p+B_1(t)y=0,& -\infty<t<t_{\Theta},\\
2^{\frac{q}{p-1}}ae^{-m\mu t}y(t)\rightarrow\gamma\ \textrm{as}\ t\rightarrow -\infty,\\
e^{-mt}(e^{-m\mu t}y(t))'\rightarrow 0\ \textrm{as}\ t\rightarrow -\infty,
\end{cases}
\end{equation}
where $t_{\Theta}:=\frac{1}{m}\log\tan\frac{\Theta}{2}$.
We define
\[
s:=t+\frac{\log\gamma}{m\mu}\ \ \textrm{and}\ \ \hat{y}(s):=y(t).
\]
Then (\ref{S5E1}) becomes
\begin{equation}\label{S5E2}
\begin{cases}
\hat{y}''+\alpha\hat{y}'-\hat{y}+\hat{y}^p+B_0(s-\frac{\log\gamma}{m\mu })\hat{y}^p+B_1(s-\frac{\log\gamma}{m\mu})\hat{y}=0, & -\infty<s<t_{\Theta}+\frac{\log\gamma}{m\mu},\\
2^{\frac{q}{p-1}}ae^{-m\mu s}\hat{y}(s)\rightarrow 1\ \textrm{as}\ s\rightarrow -\infty,\\
e^{-ms}(e^{-m\mu s}\hat{y}(s))'\rightarrow 0\ \textrm{as}\ s\rightarrow -\infty.
\end{cases}
\end{equation}
For each fixed $s$, as $\gamma\rightarrow\infty$, $B_0(s-\frac{\log\gamma}{m\mu })\rightarrow 0$ and $B_1(s-\frac{\log\gamma}{m\mu})\rightarrow 0$.
Therefore, we expect that $\hat{y}(s)$ converges to the solution of (\ref{S2E2}) in a certain sense.
\begin{lemma}\label{S5L3}
Suppose that $p>p_{\rm S}$.
Let $\bar{y}(s)$ be the solution of (\ref{S2E2}) with $\bar{\gamma}:=2^{-\frac{q}{p-1}}$.
For each $s_0\in\R$, as $\gamma\rightarrow\infty$,
\begin{align*}
\hat{y}(s)\rightarrow\bar{y}(s)\ \textrm{uniformly in}\ s\in(-\infty,s_0]\ \ \textrm{and}\ \ 
\hat{y}'(s)\rightarrow\bar{y}'(s)\ \textrm{uniformly in}\ s\in(-\infty,s_0].
\end{align*}
\end{lemma}
\begin{proof}
Multiplying the equation in (\ref{S5E1}) by $e^{m(N-2-\mu)t}$, we have
\[
\left\{ e^{m(N-2-\mu)t}(y'-m\mu y)\right\}'=-e^{m(N-2-\mu)t}(y^p+B_0(t)y^p+B_1(t)y)<0,
\]
where we use $m^2\mu(N-2-\mu)=1$.
Since
\begin{equation}\label{S5L3E1}
e^{m(N-2-\mu)t}(y'-m\mu y)\rightarrow 0\ \textrm{as}\ t\rightarrow -\infty,
\end{equation}
we see that $y'-m\mu y<0$. Since
\[
2^{\frac{q}{p-1}}ae^{-m\mu t}y(t)\rightarrow\gamma\ \textrm{as}\ t\rightarrow -\infty,
\]
we have
\[
y(t)<2^{-\frac{q}{p-1}}a^{-1}\gamma e^{m\mu t}=2^{-\frac{q}{p-1}}a^{-1}e^{m\mu s}.
\]
Since $\hat{y}(s)=y(t)$,
\begin{equation}\label{S5L3E2}
\hat{y}(s)<2^{-\frac{q}{p-1}}a^{-1}e^{m\mu s}.
\end{equation}
Multiplying the equation in (\ref{S5E2}) by $e^{m(N-2-\mu)s}$, we have
\begin{equation}\label{S5L3E3}
\left\{e^{m(N-2-\mu)s}(\hat{y}'-m\mu\hat{y})\right\}'=-e^{m(N-2-\mu)s}(\hat{y}^p+\widehat{B}_0(s)\hat{y}^p+\whB_1(s)\hat{y}),
\end{equation}
where $\whB_0(s):=B_0(s-\frac{\log\gamma}{m\mu})$ and $\whB_1(s):=B_1(s-\frac{\log\gamma}{m\mu})$.
Integrating (\ref{S5L3E3}) and solving it for $\hat{y}'$, we have
\[
\hat{y}'(s)=m\mu\hat{y}(s)-e^{-m(N-2-\mu)s}\int_{-\infty}^s\left(\hat{y}(\tau)^p+\whB_0(\tau)\hat{y}(\tau)^p+\whB_1(\tau)\hat{y}(\tau)\right)e^{m(N-2-\mu)\tau}d\tau,
\]
where we use (\ref{S5L3E1}).
Using (\ref{S5L3E2}), we have $|\hat{y}(\tau)^p+\whB_0(\tau)\hat{y}(\tau)^p+\whB_1(\tau)\hat{y}(\tau)|\le C_0e^{m\mu\tau}$, and there holds
\begin{align}
|\hat{y}'(s)|&\le m\mu|\hat{y}(s)|+e^{-m(N-2-\mu)s}\int_{-\infty}^sC_0e^{m\mu\tau}e^{m(N-2-\mu)\tau}d\tau\nonumber\\
&=m\mu|\hat{y}(s)|+\frac{C_0}{m(N-2)}e^{m\mu s}\label{S5L3E4}\\
&\le C_1e^{m\mu s}.\nonumber
\end{align}
Therefore, $\{\hat{y}(s)\}_{\gamma}$ is equicontinuous on $(-\infty,s_0]$.
It follows from the Arzel\`a-Ascoli theorem that for each fixed $s_1\in (-\infty,s_0]$, as $\gamma\rightarrow\infty$, $\hat{y}(s)$ uniformly converges to a certain function $\hat{y}_0(s)$ on $[s_1,s_0]$.
Because of (\ref{S5L3E2}), this convergence is uniform on $(-\infty,s_0]$.
By (\ref{S5L3E4}) we see that $\{\hat{y}'(s)\}_{\gamma}$ is bounded on $(-\infty,s_0]$.
Because of (\ref{S5E2}), $\{\hat{y}''(s)\}_{\gamma}$ is also bounded on $(-\infty,s_0]$.
By the same argument as before we see that as $\gamma\rightarrow\infty$, $\hat{y}'(s)$ converges to a certain function $\hat{y}_1(s)$ on $(-\infty,s_0]$.
Taking the limit of $y(s)=\int_{-\infty}^sy'(\tau)d\tau$, we see that $\hat{y}_0(s)=\int_{-\infty}^s\hat{y}_1(\tau)d\tau$, where by (\ref{S5L3E4}) we can use the dominated convergence theorem.
Hence, $\hat{y}_0(s)$ is of class $C^1$ and $\hat{y}_1=\hat{y}_0'$.
By (\ref{S5E2}) we see that $\hat{y}''(s)$ also converges to a certain function $\hat{y}_2(s)$ on $(-\infty,s_0]$.
By the same argument as before, we see that $\hat{y}_2=\hat{y}_1'(=\hat{y}_0'')$.
Taking the limit of (\ref{S5E2}), we see that $\hat{y}_0(s)$ satisfies (\ref{S2E2}) with $\bar{\gamma}=2^{-\frac{q}{p-1}}$.
Thus $\hat{y}_0=\bar{y}$.
We obtain the conclusion.
\end{proof}
Let $z(t,\gamma):=y_t(t,\gamma)$.
Then $(y,z)$ satisfies
\begin{equation}\label{S5E3}
\begin{cases}
y'=z\\
z'=-\alpha z+y-y^p-B_0(t)y^p-B_1(t)y.
\end{cases}
\end{equation}
Proposition~\ref{S2P1} (i) says that $(\bar{y}(s),\bar{z}(s))$ converges to $(1,0)$ if $p>p_{\rm S}$.
This fact and Lemma~\ref{S5L3} indicate that $(y(t,\gamma),z(t,\gamma))$ approaches to $(1,0)$ as $\gamma\rightarrow\infty$ along $t=s_0-\frac{\log\gamma}{m\mu}$ provided that $s_0$ is chosen large enough.

\begin{lemma}\label{S5L4}
Suppose that $p>p_{\rm S}$.
Let
\[
H(y,z):=\frac{z^2}{2}-\frac{y^2-1}{2}+\frac{y^{p+1}-1}{p+1},
\]
and let $\Omega_{\e}:=\{(y,z)\in\R^2;\ H(y,z)<\e,\ y>0\}$.
Then the following hold:\\
(i) Let $\e>0$ be fixed. For each large $t_0>0$, $(y(-t_0,\gamma),z(-t_0,\gamma))\in\Omega_{\e}$ provided that $\gamma>0$ is large.\\
(ii) If $(y(-t_0,\gamma),z(-t_0,\gamma))\in\Omega_{\e}$, then there is $T_{\e}<0$ independent of $t_0$ such that $(y(t,\gamma),z(t,\gamma))\in\Omega_{2\e}$ for $t\in[-t_0,T_{\e}]$.
\end{lemma}
\begin{proof}
Because of Lemma~\ref{S5L3}, for each $t_0$, as $\gamma\rightarrow\infty$,
\[
y(-t_0)=\hat{y}(s)\rightarrow\bar{y}(s)=\bar{y}(-t_0+\frac{\log\gamma}{m\mu}),
\]
where $s=-t_0+\frac{\log\gamma}{m\mu}$.
We similarly see that $z(-t_0)\rightarrow\bar{z}(-t_0+\frac{\log\gamma}{m\mu})$.
Since $(\bar{y},\bar{z})$ converges to $(1,0)$ and $\Omega_{\e}$ is a neighborhood of $(1,0)$, (i) holds.

We define $E(y,z,t)$ by
\begin{equation}\label{S5L4E0}
E(y,z,t):=H(y,z)+B_0(t)\frac{y^{p+1}}{p+1}+B_1(t)\frac{y^2}{2}.
\end{equation}
Let $y(t)$ be the solution of (\ref{S5E1}).
By direct calculation we have
\begin{equation}\label{S5L4E1}
\frac{d}{dt}E(y(t),z(t),t)=-\alpha y'(t)^2+B_0'(t)\frac{y(t)^{p+1}}{p+1}+B_1'(t)\frac{y(t)^2}{2}.
\end{equation}
Let $\xi:=\left(\frac{p+1}{2}\right)^{\frac{1}{p-1}}$.
Let $\e>0$ be small such that $\Omega_{2\e}\subset\{0\le y\le\xi\}$.
We can choose $T<0$ such that
\begin{equation}\label{S5L4E2}
B_0(T)\frac{\xi^{p+1}}{p+1}<\frac{\e}{8}\ \ \textrm{and}\ \ B_1(T)\frac{\xi^2}{2}<\frac{\e}{8}.
\end{equation}
We show that $(y(t),z(t))\in\Omega_{2\e}$ for $t\in[-t_0,T]$ if $(y(-t_0,\gamma),y_t(-t_0,\gamma))\in\Omega_{\e}$.
Suppose the contrary, i.e., we assume that
\begin{equation}\label{S5L4E2+}
(y(t),z(t))\in\Omega_{2\e}\ (-t_0\le t<T)\ \ \textrm{and}\ \ (y(T),z(T))\not\in\Omega_{2\e}.
\end{equation}
Integrating (\ref{S5L4E1}) over $[-t_0,T]$, we have
\begin{align}
E(y(T),z(T),T)-& E(y(-t_0),z(-t_0),-t_0)\nonumber\\
&\le\int_{-t_0}^T\left(B_0'(t)\frac{y(t)^{p+1}}{p+1}+B_1'(t)\frac{y(t)^2}{2}\right)dt\nonumber\\
&\le\frac{\xi^{p+1}}{p+1}\int_{-t_0}^TB_0'(t)dt+\frac{\xi^2}{2}\int_{-t_0}^TB_1'(t)dt\nonumber\\
&\le\frac{\xi^{p+1}}{p+1}B_0(T)+\frac{\xi^2}{2}B_1(T)\nonumber\\
&\le\frac{\e}{8}+\frac{\e}{8}=\frac{\e}{4},\label{S5L4E3}
\end{align}
where we use (\ref{S5L4E2}) and the two inequalities
\begin{align*}
B_0'(t)&=2mq(1+e^{2mt})^{q-1}e^{2mt}>0\ \ \textrm{for}\ \ t\in\R,\ \textrm{and},\\
B_1'(t)&=\frac{2mN(N-2)(1-e^{2mt})e^{2mt}}{(1+e^{2mt})^3}>0\ \ \textrm{for}\ \ t<0.
\end{align*}
Using (\ref{S5L4E3}) and (\ref{S5L4E0}), we have
\begin{align*}
H(y(T),z(T))&\le H(y(-t_0),z(-t_0))+B_0(-t_0)\frac{y(-t_0)^{p+1}}{p+1}+B_1(-t_0)\frac{y(-t_0)^2}{2}\\
&\quad-\left(B_0(T)\frac{y(T)^{p+1}}{p+1}+B_1(T)\frac{y(T)^2}{2}\right)+\frac{\e}{4}\\
&\le \e+\frac{\e}{8}+\frac{\e}{8}+\frac{\e}{4}=\frac{3}{2}\e.
\end{align*}
Hence, $(y(T),z(T))\in\Omega_{3\e/2}\subset\Omega_{2\e}$, which contradicts (\ref{S5L4E2+}).
The proof of (ii) is complete.
\end{proof}

\begin{proof}[Proof of Lemma~\ref{S5L2}]
Let $\{\gamma_n\}_{n=1}^{\infty}$ be a sequence diverging to $\infty$.
Let $y_n:=y(t,\gamma_n)$ be the solution of (\ref{S5E1}), and let $z_n:=y_n'$.
We fix $\e>0$.
By Lemma~\ref{S5L4} (i) we see that for arbitrary large $t_0>0$, $(y_n(-t_0),z_n(-t_0))\in\Omega_{\e}$ provided that $n$ is large.
Because of Lemma~\ref{S5L4} (ii), there is $T<0$ such that $(y_n(t),z_n(t))\in\Omega_{2\e}$ for $t\in [-t_0,T]$.
Thus, $\{(y_n(t),z_n(t))\}$ is bounded in $(C^0[-t_0,T])^2$.
It follows from the equation in (\ref{S5E1}) that $\{y_n''(t)\}$ is bounded in $C^0[-t_0,T]$.
Differentiating the equation in (\ref{S5E1}), we see that $\{z_n''(t)\}$ is also bounded in $C^0[-t_0,T]$.
Thus by the Ascoli-Arzel\`a theorem we see that $\{(y_n,z_n)\}$ converges to some pair of functions $(y_*(t),z_*(t))$ in $(C^1[-t_0,T])^2$.
Since $(y_n,z_n)$ satisfies the equation in (\ref{S5E3}), $(y_*,z_*)$ satisfies the same equation.
Next, we prove $y_*=y^*$, where $y^*$ is the solution of (\ref{S4L1E1}).
If $y_*=y^*$, then $y_n\rightarrow y^*$ in $C^2[-t_0,T]$, and $u\rightarrow u^*$ in $C^2(I)$ for some interval $I$.
Let $r_0\in I$ be fixed.
Because of the continuous dependence of $u$ in $C^2_{loc}(0,R^*]$ with respect to $(u(r_0),u'(r_0))$, $u\rightarrow u^*$ in $C^2_{loc}(0,R^*]$.
Moreover, $R(\gamma)\rightarrow R^*$.
By Lemma~\ref{S4L2} it suffices to show that
\begin{equation}\label{S5L2E0}
y_*(t)\rightarrow 1\ \textrm{as}\ t\rightarrow -\infty.
\end{equation}

We prove (\ref{S5L2E0}) by contradiction.
Suppose the contrary, i.e., there is a sequence $\{t_k\}$ such that $t_k\rightarrow -\infty$ and $(y_*(t_k),z_*(t_k))\not\in\Omega_{\delta}$ for all $k\ge 1$.
We choose $\e=\delta/4$.
By Lemma~\ref{S5L3} for each large $s_0>0$, if $\gamma$ is large, then $(\hat{y}(s_0,\gamma),\hat{y}_s(s_0,\gamma))\in\Omega_{\e}$.
Since $\hat{y}(s_0,\gamma_n)=y(t,\gamma_n)=y_n(s_0-\frac{\log\gamma_n}{m\mu})$ and $\hat{y}_s(s_0,\gamma_n)=y_t(t,\gamma_n)=z_n(s_0-\frac{\log\gamma_n}{m\mu})$, $(y_n(s_0-\frac{\log\gamma_n}{m\mu}),z_n(s_0-\frac{\log\gamma_n}{m\mu}))\in\Omega_{\e}$ provided that $n$ is large.
By Lemma~\ref{S5L4} (ii) we see that
\[
(y_n(t),z_n(t))\in\Omega_{2\e}\subset\Omega_{\delta}\ \ \textrm{for}\ \ t\in[s_0-\frac{\log\gamma_n}{m\mu},T_{\e}],
\]
where $T_{\e}$ is independent of $n$.
Since $s_0-\frac{\log\gamma_n}{m\mu}\to -\infty$ $(n\to\infty)$, we can choose $n$ such that $[s_0-\frac{\log\gamma_n}{m\mu},T_{\e}]$ includes an element of $\{t_k\}$.
We obtain a contradiction.
\end{proof}

\section{Uniqueness of a small solution}
Let $u(r,\gamma)$ be the solutions of (\ref{DD}), and let $R(\gamma)$ be the first positive zero of $u(\,\cdot\,,\gamma)$.
\begin{lemma}\label{S6L1}
Suppose that $p>1$. Then
\[
R(\gamma)\rightarrow\infty\ \textrm{as}\ \gamma\downarrow 0.
\]
\end{lemma}
\begin{proof}
Let $v(r):=A(r)^{-\frac{N-2}{2}}u(r)$.
Then, $v$ satisfies
\[
(r^{N-1}A(r)^{N-2}v')'+r^{N-1}A(r)^Nv^p=0.
\]
Integrating this equation over $[0,r]$, we have
\[
v'(r)=-\frac{1}{r^{N-1}A(r)^{N-2}}\int_0^rs^{N-1}A(s)^Nv^pds\le 0.
\]
where we use $v'(0)=0$.
Let $\delta:=2^{-\frac{N-2}{2}}\gamma$.
Since $0\le v\le\delta$,
\[
-v'(r)\le\frac{1}{r^{N-1}A(r)^{N-2}}\int_0^rs^{N-1}A(s)^N\delta^pds,
\]
and hence,
\begin{equation}\label{S6L1E1}
-\frac{v'(r)}{2\delta^p}\le\frac{1}{r^{N-1}A(r)^{N-2}}\int_0^r(sA(s))^{N-1}ds.
\end{equation}
We have
\[
\int_0^r(sA(s))^{N-1}ds\le
\begin{cases}
\frac{2^{N-1}}{N}+\int_1^r(sA(s))ds=C_0+\log (1+r^2), & 1\le r,\\
\int_0^r(2s)^{N-1}ds=\frac{2^{N-1}r^N}{N}, & 0\le r\le 1.
\end{cases}
\]
Integrating (\ref{S6L1E1}) over $[0,R(\gamma)]$, we have
\begin{align*}
-\int_0^{R(\gamma)}\frac{v'(r)}{2\delta^p}&\le\int_0^{R(\gamma)}\frac{1}{r^{N-1}A(r)^{N-2}}\int_0^r(sA(s))^{N-1}dsdr\\
&\le\int_0^1\frac{2^{N-1}r}{NA(r)^{N-2}}dr+\int_1^{R(\gamma)}\frac{C_0+\log(r^2+1)}{r^{N-1}A(r)^{N-2}}dr\ \ \textrm{for}\ \ R>1.
\end{align*}
The first positive zero of $v(\,\cdot\,)$ is equal to that of $u(\,\cdot\,)$, i.e., $R(\gamma)$.
Therefore, $v(R(\gamma))=0$.
Since $v(0)=\delta$ and $C_1:=\int_0^1\frac{2^{N-1}r}{NA(r)^{N-2}}dr<\infty$,
\begin{equation}\label{S6L1E3}
\frac{1}{2\delta^{p-1}}\le C_1+\int_1^{R(\gamma)}\frac{C_0+\log(1+r^2)}{r^{N-1}A(r)^{N-2}}dr\ \ \textrm{for}\ \ R(\gamma)>1.
\end{equation}
Taking the limit $\delta\downarrow 0$, we see that the right-hand side of (\ref{S6L1E3}) diverges.
Hence, $R(\gamma)\rightarrow\infty$ as $\delta\downarrow 0$.
\end{proof}

\begin{lemma}\label{S6L2}
Suppose that $p>1$.
There is a $\gamma_0>0$ such that $R'(\gamma)<0$ for $\gamma\in(0,\gamma_0)$.
In particular, if $\gamma\in(0,\gamma_0)$, then $u(r,\gamma)$ is nondegenerate in the space of radial functions.
\end{lemma}

\begin{proof}
By $\calL$ we denote
\[
\calL:=\frac{d^2}{dr^2}+\frac{N-1}{r}\frac{d}{dr}+\frac{N(N-2)}{4}A(r)^2.
\]
We define $w(r):=u_{\gamma}(r,\gamma)$.
Then $w(r)$ satisfies
\begin{equation}\label{S6L2E1}
\begin{cases}
\left(\calL+{pA(r)^{-q}u(r,\gamma)^{p-1}}\right)w=0, & 0<r<R(\gamma),\\
w(0)=1,\ w'(0)=0.\\
\end{cases}
\end{equation}
We show that
\begin{equation}\label{S6L2E2}
w(R(\gamma))<0.
\end{equation}
Let $\psi_0(r):=A(r)^{\frac{N-2}{2}}(A(r)-1)$.
Then, by direct calculation we see that $\psi_0(r)$ satisfies
\[
\begin{cases}
\left(\calL+NA(r)^2\right)\psi_0=0, & 0<r<\infty,\\
\psi_0(0)=2^{\frac{N-2}{2}},\ \psi_0'(0)=0.\\
\end{cases}
\]
Note that $\psi_0$ has a unique zero at $r=1$ on $[0,\infty)$ and that $\psi_0$ corresponds to the second eigenfunction of $\Delta_{\bbS^N}$ on the whole sphere.
Since $U(\theta)(=A(r)^{-\frac{N-2}{2}}u(r))$ satisfies (\ref{EFODE}), $U(\theta)$ is decreasing and $|U(\theta)|\le\Gamma$ $(0\le\theta\le\Theta)$, where $\Gamma=2^{-\frac{N-2}{2}}\gamma$.
Therefore, $|u(r)|\le 2^{-\frac{N-2}{2}}\gamma A(r)^{\frac{N-2}{2}}$ for $r\in[0,R(\gamma)]$.
We have
\[
{p{A(r)^{-q}}u(r,\gamma)^{p-1}}\le 2^{-\frac{(N-2)(p-1)}{2}}p\gamma^{p-1}A(r)^2\ \ \textrm{for}\ \ r\in[0,R(\gamma)].
\]
Thus, if $\gamma>0$ is small, then ${p{A(r)^{-q}}u(r,\gamma)^{p-1}}\le NA(r)^2$ for $r\in[0,R(\gamma)]$.
Hence, by the oscillation theorem for Sturm-Liouville equations (e.g., see Ince~\cite[pp.224--225]{I44}) we see that $w(r)$ oscillates more slowly than $\psi_0(r)$.
Since $R(\gamma)$ is large, $\psi_0(r)$ has exactly one zero on $[0,R(\gamma)]$, and hence $w(r)$ has at most one zero on $[0,R(\gamma)]$.
Let $\lambda_1$ be the first eigenvalue of the eigenvalue problem
\begin{equation}\label{S6L2E3}
\begin{cases}
\left(\calL+{p{A(r)^{-q}}u(r,\gamma)^{p-1}}\right)\phi+\lambda\phi=0, & 0<r<R(\gamma),\\
\phi(R(\gamma))=0,\\
\phi(r)>0, & 0\le r<R(\gamma),\\
\phi'(0)=0.
\end{cases}
\end{equation}
We define
\[
\calH[\psi]:=\int_0^{R(\gamma)}\left((\psi')^2-\frac{N(N-2)}{4}A(r)\psi^2-\frac{pu(r,\gamma)^{p-1}}{A(r)^q}\psi^2\right)r^{N-1}dr.
\]
Multiplying $\left(\calL+{p{A(r)^{-q}}u(r,\gamma)^{p-1}}\right)u=(p-1)u^p$ by $ur^{N-1}$ and integrating it, we have
\begin{equation}\label{S6L2E3+}
\calH[u]=-(p-1)\int_0^{R(\gamma)}u(r,\gamma)^{p+1}r^{N-1}dr<0.
\end{equation}
Using a variational characterization of $\lambda_1$ and (\ref{S6L2E3+}), we have
\begin{equation}\label{S6L2E4}
\lambda_1=\inf_{\psi\in X}\frac{\calH[\psi]}{\left\|\psi\right\|^2_{L^2}}\le\frac{\calH[u]}{\left\|u\right\|^2_{L^2}}<0,
\end{equation}
where $\left\|\psi\right\|_{L^2}:=\left(\int_0^{R(\gamma)}\psi^2r^{N-1}dr\right)^{1/2}$ and
\[
X:=\left\{\psi(r);\ \int_0^{R(\gamma)}\left((\psi')^2+\psi^2\right)r^{N-1}dr<\infty\ \ \textrm{and}\ \ \psi(R(\gamma))=0\right\}.
\]
The first eigenfunction $\phi_1(r)$ satisfies
\[
\begin{cases}
\left(\calL+{p{A(r)^{-q}}u(r,\gamma)^{p-1}}+\lambda_1\right)\phi_1=0, & 0<r<R(\gamma),\\
\phi_1(0)=1,\ \phi_1'(0)=0.\\
\phi_1>0, & 0\le r<R(\gamma).\\
\end{cases}
\]
Since ${p{A(r)^{-q}}u(r,\gamma)^{p-1}}+\lambda_1<{p{A(r)^{-q}}u(r,\gamma)^{p-1}}$, by the oscillation theorem we see that $w(r)$ oscillates more rapidly than $\phi_1(r)$, and hence $w(r)$ has at least one zero on $[0,R(\gamma)]$.
Thus $w(r)$ has exactly one zero on $[0,R(\gamma)]$.
If $w(R(\gamma))=0$, then $w(r)>0$ on $[0,R(\gamma))$.
Therefore, $0$ is the first eigenvalue, which contradicts (\ref{S6L2E4}).
Thus, $w(R(\gamma))\neq 0$.
Since $w(0)>0$, $w(r)$ has exactly one zero on $(0,R(\gamma))$, which indicates that $w(R(\gamma))<0$.
We obtain (\ref{S6L2E2}).

Next, we prove the statements of the lemma, using (\ref{S6L2E2}).
Differentiating $u(R(\gamma),\gamma)=0$ in $\gamma$, we have $u_r(R(\gamma),\gamma)R'(\gamma)+u_{\gamma}(R(\gamma),\gamma)=0$.
It follows from Hopf's boundary point lemma that $u_r(R(\gamma),\gamma)<0$.
Hence,
\[
R'(\gamma)=-\frac{u_{\gamma}(R(\gamma),\gamma)}{u_r(R(\gamma),\gamma)}<0.
\]
Because of (\ref{S6L2E1}), $0$ is an eigenvalue of (\ref{S6L2E3}) if and only if $w(R(\gamma))=0$.
By (\ref{S6L2E2}) we see that $0$ is not an eigenvalue which means that $u(r,\gamma)$ is nondegenerate.
\end{proof}

\begin{remark}
In the above proof we show that $w(r)$ has one zero in $(0,R(\gamma))$ and $w(R(\gamma))<0$.
This indicates that the Morse index of $u$ in the space of radial functions is one.
We do not use this fact in this paper.
\end{remark}

\section{Infinitely many turning points}
First, we show that $R(\gamma)$ oscillates around $R^*$ as $\gamma\to\infty$.
\begin{lemma}\label{S7L1}
Suppose that $p_{\rm S}<p<p_{\rm JL}$.
Let $u(r,\gamma)$ be the solution of (\ref{DD}), and let $R(\gamma)$ be the first positive zero of $u(\,\cdot\,,\gamma)$.
Let $(R^*,u^*(r))$ be the singular solution of (\ref{D}) given in Remark~\ref{rem1}.
Then the following hold:\\
(i) $\calZ_{(0,\min\{R(\gamma),R^*\})}[u(\,\cdot\,,\gamma)-u^*(\,\cdot\,)]\rightarrow\infty$ as $\gamma\rightarrow\infty$.\\
(ii) $R(\gamma)$ oscillates around $R^*$ infinitely many times as $\gamma\to\infty$.
\end{lemma}
\begin{proof}
We prove (i), using a blow-up argument.
We change variables
\[
\tilde{u}(\rho,\gamma):=2^{-\frac{q}{p-1}}\gamma^{-1}u(r,\gamma)\ \ \textrm{and}\ \ \rho:=\gamma^{\frac{p-1}{2}}r.
\]
Then $\tilde{u}(\rho)$ satisfies
\begin{equation}\label{S7L1E1}
\begin{cases}
\tu''+\frac{N-1}{\rho}\tu'+\tu^p+\wtB_0(\rho,\gamma)\tu^p+\wtB_1(\rho,\gamma)\tu=0, & 0<\rho<\wtR(\gamma),\\
\tu(0)=1,\ \tu'(0)=0,\\
\end{cases}
\end{equation}
where $\wtR(\gamma):=\gamma^{\frac{p-1}{2}}R(\gamma)$ which is the first positive zero of $\tu(\,\cdot\,,\gamma)$,
\[
\wtB_0(\rho,\gamma):=\left(1+\frac{\rho^2}{\gamma^{p-1}}\right)^q-1,\ \ \textrm{and}\ \ \wtB_1(\rho,\gamma):=\frac{N(N-2)\gamma^{p-1}}{(\gamma^{p-1}+\rho^2)^2}.
\]
From Lemma~\ref{S5L2} it holds $\wtR(\gamma)\to\infty$ $(\gamma\to\infty)$.
Let $\rho_0>0$ be large.
If $\gamma$ is large, then the interval $[0,\rho_0]$ is included in $[0,\wtR(\gamma)]$.
Since $\wtB_0>0$ and $\wtB_1>0$, it is clear from the equation in (\ref{S7L1E1}) that $\tu(\rho)$ is decreasing on $[0,\wtR(\gamma)]$.
Therefore, $0\le\tu(\rho)\le 1$ on $[0,\rho_0]$ provided that $\gamma$ is large.
Since $|\wtB_0(\rho,\gamma)|$ and $|\wtB_1(\rho,\gamma)|$ uniformly converge to $0$ on $[0,\rho_0]$, $|\wtB_0(\rho,\gamma)\tu(\rho)^p|+|\wtB_1(\rho,\gamma)\tu(\rho)|\to 0$ in $C^0[0,\rho_0]$.
It follows from the equation in (\ref{S7L1E1}) that as $\gamma\to\infty$,
\begin{equation}\label{S7L1E2}
\tu(\rho)\to\bu(\rho)\ \ \textrm{in}\ \ C^1[0,\rho_0],
\end{equation}
where $\bu(\rho)$ is the solution of (\ref{S2E1}) with $\bar{\gamma}=1$.
Next, we apply the same change of variables to the singular solution $u^*(r)$. We define $\tu^*(\rho)$ by
\[
\tu^*(\rho):=2^{-\frac{q}{p-1}}\gamma^{-1}u^*(r)\ \ \textrm{and}\ \ \rho=\gamma^{\frac{p-1}{2}}r.
\]
By (\ref{S4C2E1}) we have
\begin{equation}\label{S7L1E3}
\tu^*(\rho)=a\rho^{-\frac{2}{p-1}}(1+o(1))\ \textrm{as}\ \frac{\rho}{\gamma^{\frac{p-1}{2}}}\to 0.
\end{equation}
When $\rho\in[0,\rho_0]$, $\frac{\rho}{\gamma^{\frac{p-1}{2}}}$ uniformly converges to $0$, and hence $o(1)$ in (\ref{S7L1E3}) uniformly converges to $0$.
Since $\tu^*(\rho)$ is unbounded near $\rho=0$,
\[
\tu^*(\rho)\to\bu^*(\rho)\ \textrm{in}\ C^0_{loc}(0,\rho_0],
\]
where $\bu^*(\rho)$ is defined by (\ref{S2E0}).
Since $\tu^*(\rho)$ satisfies the ODE in (\ref{S7L1E1}), this convergence holds in $C^2_{loc}(0,\rho_0]$, i.e.,
\begin{equation}\label{S7L1E4}
\tu^*(\rho)\to\bu^*(\rho)\ \textrm{in}\ C^2_{loc}(0,\rho_0].
\end{equation}
On the other hand, if $\gamma>0$ is large, then
\begin{align}
\calZ_{(0,\min\{R(\gamma),R^*\})}[u(\,\cdot\,,\gamma)-u^*(\,\cdot\,)]&=\calZ_{(0,\min\{\wtR(\gamma),\gamma^{\frac{p-1}{2}}R^*\})}[\tu(\,\cdot\,,\gamma)-\tu^*(\,\cdot\,)]\nonumber\\
&\ge\calZ_{(0,\rho_0)}[\tu(\,\cdot\,,\gamma)-\tu^*(\,\cdot\,)].\label{S7L1E5}
\end{align}
We see by (\ref{S7L1E2}) and (\ref{S7L1E4}) that if $\gamma>0$ is large, then
\[
\calZ_{(0,\rho_0)}[\tu(\,\cdot\,,\gamma)-\tu^*(\,\cdot\,)]\ge\calZ_{(0,\rho_0)}[\bu(\,\cdot\,)-\bu^*(\,\cdot\,)].
\]
Proposition~\ref{S2P1} (ii) says that $\calZ_{(0,\rho_0)}[\bu(\,\cdot\,)-\bu^*(\,\cdot\,)]\to\infty$ as $\rho_0\to\infty$.
Therefore, if $\rho_0$ and $\gamma$ are large and $\rho_0\le\wtR(\gamma)$, then $\calZ_{(0,\rho_0)}[\tu(\,\cdot\,,\gamma)-\tu^*(\,\cdot\,)]$ can be arbitrary large.
By (\ref{S7L1E5}) we see that (i) holds.

We prove (ii).
Since $u(r,\gamma)$ and $u^*(r)$ satisfy the same equation, every zero of $u(\,\cdot\,,\gamma)-u^*(\,\cdot\,)$ is simple.
Each zero continuously depends on $\gamma$.
The zero number of $u(\,\cdot\,,\gamma)-u^*(\,\cdot\,)$ on a bounded interval is finite, since the zero set of $u(\,\cdot\,,\gamma)-u^*(\,\cdot\,)$ does not have an accumulation point.
Let $I(\gamma):=(0,\min\{R(\gamma),R^*\})$.
The intersection number $\calZ_{I(\gamma)}[u(\,\cdot\,,\gamma)-u^*(\,\cdot\,)]$ is preserved if another zero does not come from $\partial I(\gamma)$.
Since $u(0,\gamma)-u^*(0)=-\infty$, a zero cannot come from $0\in\partial I(\gamma)$.
If $R(\gamma)>R^*$ for large $\gamma$, then there is $C>0$ such that $\calZ_{I(\gamma)}[u(\,\cdot\,,\gamma)-u^*(\,\cdot\,)]\le C$ for all $\gamma>0$, which contradicts (i).
If $R(\gamma)<R^*$ for large $\gamma$, then we similarly obtain a contradiction.
Therefore, there are a positive integer $m$ and a sequence $\{\gamma_n\}_{n=m}^{\infty}$ $(\gamma_m<\gamma_{m+1}<\cdots\to\infty)$ such that $\calZ_{I(\gamma_n)}[u(\,\cdot\,,\gamma)-u^*(\,\cdot\,)]=n$ and $u(\,\cdot\,,\gamma)-u^*(\,\cdot\,)$ has a zero at $\min\{R(\gamma),R^*\}$, i.e., $R(\gamma)=R^*$.
Since the zero set is discrete, there is a sequence $\{\hat{\gamma}_n\}_{n=m}^{\infty}$ such that $\gamma_m<\hat{\gamma}_m<\gamma_{m+1}<\hat{\gamma}_{m+1}<\cdots$ and $R(\hat{\gamma}_n)\neq R^*$.
We easily see the following:
If $\calZ_{I(\gamma)}[u(\,\cdot\,,\gamma)-u^*(\,\cdot\,)]$ is even (resp. odd), then $R(\hat{\gamma}_n)<R^*$ (resp. $R^*<R(\hat{\gamma}_n)$).
Thus, (ii) holds.
\end{proof}

\begin{proof}[Proof of Theorem~\ref{A}]
Let $\Theta(\Gamma):=2\arctan R(\gamma)$, $\Gamma:=2^{-\frac{N-2}{2}}\gamma$, and $\Theta^*:=2\arctan R^*$.
Note that the range of $\Theta$ is $(0,\pi)$.
Lemma~\ref{S3L3} says that $R(\gamma)$ is a $C^1$-function on $(0,\infty)$ and $0<R(\gamma)<\infty$ for $\gamma\in(0,\infty)$.
Hence, (i) holds.
It follows from Lemma~\ref{S6L1} that $\Theta(\Gamma)\to\pi$ as $\Gamma\downarrow 0$.
Since $\Theta'(\Gamma)=2^{\frac{N}{2}}R'(\gamma)/(1+R(\gamma)^2)$, we see by Lemma~\ref{S6L2} that $\Theta'(\Gamma)<0$ if $\Gamma>0$ is small.
Thus, (ii) holds.
By Lemma~\ref{S5L2} we see that $\Theta(\Gamma)\to\Theta^*$ as $\Gamma\to\infty$.
Thus, (iii) holds.
By Lemma~\ref{S7L1}~(ii) we see that (iv) holds.
The proof is complete.
\end{proof}

\begin{proof}[Proof of Corollary~\ref{B}]
Let $\underline{\Theta}:=\inf\{\Theta(\Gamma);\ \Gamma>0\}$.
Since $\Theta(\Gamma)\to\Theta^*$ $(\Gamma\to\infty)$, $\Theta(\Gamma)\to\pi$ $(\Gamma\to 0)$, and $\Theta(\Gamma)$ is continuous, we see that $\underline{\Theta}>0$.
Therefore, (i) holds.
If $p_{\rm S}<p<p_{\rm JL}$, then $\Theta(\Gamma)$ oscillates around $\Theta^*$.
Hence, $\underline{\Theta}<\Theta^*$ and $\{\Gamma>0;\ \Theta(\Gamma)\le\Theta^*-\e\}$ is bounded for small $\e>0$.
The infimum is attained, and (ii) holds.
(iii) follows from Theorem~\ref{A}~(iv).
If $\Gamma_0>0$ is small, then $\Theta'(\Gamma)<0$ for $\Gamma\in(0,\Gamma_0)$, because of Theorem~\ref{A}~(ii).
On the other hand, $\Theta_0:=\sup_{\Gamma\ge\Gamma_0}\Theta(\Gamma)<\pi$, because of Theorem~\ref{A}~(iii).
We see that if $\Theta_1\in(\frac{\Theta_0+\pi}{2},\pi)$, then there exists the unique $\Gamma>0$ such that $\Theta(\Gamma)=\Theta_1$ and $0<\Gamma<\Gamma_0$.
It is known that the solution $(\Theta(\Gamma),U(\theta))$ is nondegenerate if and only if $\Theta'(\Gamma)\neq 0$ which is equivalent to $U'(\Theta(\Gamma))\neq 0$.
The nondegeneracy holds, since $\Theta'(\Gamma)\neq 0$ for $\Gamma\in(0,\Gamma_0)$.
Thus, (iv) holds.
\end{proof}

\section{Asymptotic shapes of the branch as $p\to\infty$ and $p\downarrow 1$}
We briefly prove Proposition~\ref{CriSub} before proving Theorems~\ref{ThD} and \ref{ThE}.
\begin{proof}[Proof of Proposition~\ref{CriSub}]
Since Lemmas~\ref{S3L3} and \ref{S6L1} hold for $p=p_{\rm S}$, (i) and (ii) hold.
Shioji-Watanabe~\cite[Theorem 5]{SW13} showed that if $N\ge 3$ and $1<p\le p_{\rm S}$, then (\ref{EFODE}) has at most one solution.
Since $\Theta(\Gamma)$ is continuous and $\Theta(\Gamma)\to\pi$ $(\Gamma\downarrow 0)$, $\Theta(\Gamma)$ should be strictly decreasing, otherwise (\ref{EFODE}) has more than two solutions, which is a contradiction.
Thus, (iii) holds.
When $N\ge 4$ and $p=p_{\rm S}$, Bandle {\it et al.}~\cite[Section~7.4]{BBF98} showed that for each $\Theta\in(0,\pi)$, (\ref{EFODE}) has a regular solution.
This result indicates that $\Theta(\Gamma)\to 0$ $(\Gamma\to\infty)$, otherwise $\Theta(\Gamma)\to c>0$ and (\ref{EFODE}) has no solution for $\Theta\in (0,c)$, which is a contradiction.
Thus, (iv) holds.
When $N=3$ and $p=p_{\rm S}$,  Bandle-Peletier~\cite[Theorem 1]{BP99} showed that (\ref{EFODE}) has no regular solution for $\Theta\in(0,\frac{\pi}{2}]$ and that it has a regular solution for $\Theta\in(\frac{\pi}{2},\pi)$.
This indicates that $\Theta(\Gamma)\downarrow\frac{\pi}{2}$ as $\Gamma\to\infty$.
When $N=3$ and $1<p<p_{\rm S}$, it is easily shown that (\ref{EFS}) has a radial solution for each $\Theta\in(0,\pi)$.
This indicates that $\Theta(\Gamma)\to 0$ as $\Gamma\to\infty$.
Hence, (v) holds.
\end{proof}

\begin{proof}[Proof of Theorem~\ref{ThD}]
Let $U(\theta)$ be the solution of (\ref{EFODE}).
Then, $U(\Theta)=0$ and $U(\theta)$ is a solution of (\ref{IVP}) for some $\Gamma>0$.
We use the Poho\v{z}aev identity of the following type:
\begin{multline}\label{S8P1E-1}
H(\theta):=-U'(\theta)^2\sin^{2N-2}\theta\int_{\theta}^{\Theta}\frac{d\varphi}{\sin^{N-1}\varphi}-U(\theta)U'(\theta)\sin^{N-1}\theta\\
-\frac{2}{p+1}U(\theta)^{p+1}\sin^{2N-2}\theta\int_{\theta}^{\Theta}\frac{d\varphi}{\sin^{N-1}\varphi}.
\end{multline}
It is clear that
\begin{equation}\label{S8P1E0}
H(\Theta)=0.
\end{equation}
By L'Hopital's rule we have
\begin{equation}\label{S8P1E1}
\lim_{\theta\downarrow 0}\frac{\int_{\theta}^{\Theta}\frac{d\varphi}{\sin^{N-1}\varphi}}{\frac{1}{\sin^{N-2}\theta}}
=\lim_{\theta\downarrow 0}\frac{-\sin^{-N+1}\theta}{(-N+2)\sin^{-N+1}\theta}\nonumber\\
=\frac{1}{N-2}.
\end{equation}
By (\ref{S8P1E1}) we have
\begin{align}
\lim_{\theta\downarrow 0}\sin^{2N-2}\theta\int_{\theta}^{\Theta}\frac{d\varphi}{\sin^{N-1}\varphi}&=\lim_{\theta\downarrow 0}\sin^N\theta\left(\sin^{N-2}\theta\int_{\theta}^{\Theta}\frac{d\varphi}{\sin^{N-1}\varphi}\right)\nonumber\\
&=0.\label{S8P1E2}
\end{align}
Using (\ref{S8P1E2}), we have
\begin{equation}\label{S8P1E3}
\lim_{\theta\downarrow 0}H(\theta)=0.
\end{equation}
Differentiating $H(\theta)$ in $\theta$, we have
\[
H'(\theta)=\frac{4N-4}{p+1}U(\theta)^{p+1}\sin^{N-1}\theta\left(\frac{p+3}{4N-4}-F(\theta)\right),
\]
where
$F(\theta):=\cos\theta\sin^{N-2}\theta\int_{\theta}^{\Theta}\frac{d\varphi}{\sin^{N-1}\varphi}$.
Hereafter, let $\Theta=\Theta_0\in(0,\pi)$ be fixed.
By (\ref{S8P1E1}) we have
\begin{equation}\label{S8P1E4}
\lim_{\theta\downarrow 0}F(\theta)=\frac{1}{N-2}.
\end{equation}
Because of (\ref{S8P1E4}) and the continuity of $F(\theta)$ on $(0,\Theta_0]$, we see that $\sup_{0<\theta\le\Theta_0}F(\theta)<\infty$.
Therefore there is a large $\bar{p}=\bar{p}(\Theta_0)>0$ such that if $p>\bar{p}$, then $H'(\theta)>0$ for $\theta\in(0,\Theta_0)$.
We obtain a contradiction, because of (\ref{S8P1E0}) and (\ref{S8P1E3}).
Thus, if $p>\bar{p}$, then (\ref{EFODE}) has no solution for $\Theta=\Theta_0$.
Since the solution set $\{(\Theta(\Gamma),\Gamma)\}$ is a continuous curve including a point near $(\pi,0)$, (\ref{EFODE}) has no solution for $\Theta\in(0,\Theta_0]$.
We prove the first statement of Theorem~\ref{D} by contradiction.
Suppose the contrary, i.e., there is $\e>0$ such that $\underline{\Theta}\in(0,\pi-\e)$ for large $p>1$, where $\underline{\Theta}$ is given in Corollary~\ref{B} (i).
Let $\Theta_1:=\pi-\frac{\e}{2}(>\underline{\Theta})$.
If $p>\bar{p}(\Theta_1)$, then (\ref{EFODE}) has no solution for $\Theta=\Theta_1$.
This is a contradiction, because the definition of $\underline{\Theta}$ says that (\ref{EFODE}) has a solution for $\Theta\in(\underline{\Theta},\pi)$.
Thus, $\underline{\Theta}\to\pi$ as $p\to\infty$.

We consider the case $N=3$.
Then,
\[
F(\Theta)=\cos\theta\sin\theta\int_{\theta}^{\Theta}\frac{d\varphi}{\sin^2\varphi}\\
=\frac{1}{2}-\frac{\sin(2\theta-\Theta)}{\sin\Theta}.
\]
When $N=3$, we have
\begin{align*}
\frac{p+3}{4N-4}-F(\theta)&=\frac{p-1}{8}+\frac{\sin(2\theta-\Theta)}{2\sin\Theta}\\
&>\frac{p-1}{8}-\frac{1}{2\sin\Theta}\ \ \textrm{for}\ \ \theta\in[0,\pi]\backslash\left\{\frac{\Theta}{2}+\frac{3+4n}{4}\pi;\ n\in\Z\right\}.
\end{align*}
Therefore, if $\sin\Theta\ge\frac{4}{p-1}$, then (\ref{EFODE}) has no solution.
Since this nonexistence result is valid for $p\ge 5(=p_{\rm S})$, we assume hereafter that $p\ge p_{\rm S}$.
Since the solution set is a continuous curve and it includes a point near $(\pi,0)$, (\ref{EFODE}) has no solution if $\Theta\le\pi-\arcsin\frac{4}{p-1}$.
Thus, $\underline{\Theta}\ge\pi-\arcsin\frac{4}{p-1}$ for $p\ge p_{\rm S}$.
\end{proof}

We consider the case $p=1$.
First, we investigate the following eigenvalue problem:
\begin{equation}\label{S8E1}
\begin{cases}
\phi''+(N-1)\frac{\cos\theta}{\sin\theta}\phi'+\lambda\phi=0, & 0<\theta<\Theta,\\
\phi(0)=1,\ \phi'(0)=0,\\
\phi(\Theta)=0.
\end{cases}
\end{equation}
\begin{lemma}\label{S8L1}
Let $\lambda_1(\Theta)$ be the first eigenvalue of (\ref{S8E1}).
Then, $\lambda_1(\Theta)$ is continuous and strictly decreasing, $\lambda_1(\Theta)\to 0$ as $\Theta\uparrow\pi$, and $\lambda_1(\Theta)\to\infty$ as $\Theta\downarrow 0$.
In particular, for $N=3$, $\lambda_1(\Theta)=(\frac{\pi}{\Theta})^2-1$.
\end{lemma}
\begin{proof}
First, we consider the case $N=3$.
Let $\bar{\phi}(\theta):=\phi(\theta)\sin\theta$.
Then, $\bar{\phi}$ satisfies
\[
\begin{cases}
\bar{\phi}''+(1+\lambda)\bar{\phi}=0, & 0<\theta<\Theta,\\
\bar{\phi}(0)=\bar{\phi}(\Theta)=0.
\end{cases}
\]
Thus, $\bar{\phi}(\theta)=c\sin\frac{\pi\theta}{\Theta}$ for some $c\in\R$ and $1+\lambda=(\frac{\pi}{\Theta})^2$.
Since $\phi(0)=1$, $c$ is equal to $\frac{\Theta}{\pi}$ and $\phi(\theta)=\frac{\Theta\sin\frac{\pi\theta}{\Theta}}{\pi\sin\theta}$.
Since
$\phi'(0)=\lim_{\theta\downarrow 0}\frac{\frac{\Theta\sin\frac{\pi\theta}{\Theta}}{\pi\sin\theta}-1}{\theta}=0$
and $\phi(\theta)>0$ on $[0,\Theta)$, $\phi$ satisfies (\ref{S8E1}) and $\phi$ is the first eigenfunction.
Therefore, $\lambda_1(\Theta)=(\frac{\pi}{\Theta})^2-1$.

Next, we consider the case $N\ge 4$.
By a similar method as in the proof of Lemma~\ref{S3L2} we can prove that, for each $\lambda>0$, there exists $\Theta=\Theta(\lambda)\in (0,\pi)$ such that (\ref{S8E1}) holds.
Let $\phi(\theta,\lambda)$ be the solution of (\ref{S8E1}).
Then $\phi$ is of class $C^1$.
It follows from the uniqueness of the solution of (\ref{S8E1}) that $\phi_{\theta}(\Theta,\lambda)\neq 0$.
Applying the implicit function theorem to $\phi(\theta,\lambda)=0$, we see that $\Theta(\lambda)$, which satisfies $\phi(\Theta(\lambda),\lambda)=0$, is of class $C^1$.

On the other hand, by Theorem~III in \cite{I44}, for each $\Theta_1\in(0,\pi)$, there exists the first eigenvalue $\lambda_1>0$ such that $\Theta(\lambda_1)=\Theta_1$.
By the Sturm-Liouville comparison theorem, if $\lambda_a<\lambda_b$, then $\Theta(\lambda_b)<\Theta(\lambda_a)$, which indicates that $\Theta(\lambda)$ is strictly decreasing.
Thus, the inverse function $\lambda_1=\lambda_1(\Theta)$ exists and it is continuous and strictly decreasing.
Let $\Theta_0\in(0,\pi)$ be fixed.
Then, as $\lambda\to 0$, $\phi(\theta)$ converges to $\phi_*(\theta)$ uniformly on $[0,\Theta_0]$, where $\phi_*$ is the unique solution of the problem
\[
\begin{cases}
\phi_*''+(N-1)\frac{\cos\theta}{\sin\theta}\phi_*'=0, & 0<\theta<\Theta_0,\\
\phi_*(0)=1,\ \phi_*'(0)=0.
\end{cases}
\]
It is clear that $\phi_*(\theta)\equiv 1$.
For each $\Theta_0\in(0,\pi)$, the solution of (\ref{S8E1}) satisfies that $\phi(\theta)>0$ on $[0,\Theta_0]$ for small $\lambda>0$.
We can choose $\Theta_0$ arbitrarily close to $\pi$.
Hence, $\Theta(\lambda)\uparrow\pi$ as $\lambda\to 0$ which indicates that
\begin{equation}\label{S8L1E1}
\lambda_1(\Theta)\to 0\ \textrm{as}\ \Theta\uparrow\pi.
\end{equation}

We consider the initial value problem
\[
\begin{cases}
\phi''+(N-1)\frac{\cos\theta}{\sin\theta}\phi'+\lambda\phi=0, & 0<\theta<\pi,\\
\phi(0)=1,\ \phi'(0)=0.
\end{cases}
\]
We use the same change of variables as in Section~1.
Let $\psi(r):=A(r)^{\frac{N-2}{2}}\phi(\theta)$ and $r:=\tan\frac{\theta}{2}$.
Then $\psi(r)$ satisfies
\[
\begin{cases}
\psi''+\frac{N-1}{r}\psi'+\frac{N(N-2)}{4}A(r)^2\psi+\lambda A(r)^2\psi=0, & 0<r<\infty,\\
\psi(0)=2^{\frac{N-2}{2}},\ \psi'(0)=0.
\end{cases}
\]
Let $\tpsi(s):=\psi(r)$ and $s:=2\sqrt{\lambda}r$.
Then $\tpsi(s)$ satisfies
\[
\begin{cases}
\tpsi''+\frac{N-1}{s}\tpsi'+\frac{N(N-2)}{4\lambda}\left(\frac{\lambda}{\lambda+s^2}\right)^2\tpsi+\left(\frac{\lambda}{\lambda+s^2}\right)^2\tpsi=0, & 0<s<\infty,\\
\tpsi(0)=2^{\frac{N-2}{2}},\ \tpsi'(0)=0.
\end{cases}
\]
Taking the limit $\lambda\to\infty$, we see that $\tpsi(s)$ converges to $\tpsi_*(s)$ uniformly on any bounded interval, where $\tpsi_*(s)$ is the solution of
\[
\begin{cases}
\tpsi''_*+\frac{N-1}{s}\tpsi'_*+\tpsi_*=0, & 0<s<\infty,\\
\tpsi_*(0)=2^{\frac{N-2}{2}},\ \tpsi'_*(0)=0.
\end{cases}
\]
Moreover, $\tpsi_*$ can be explicitly written as $\tpsi_*(s)=cs^{-\frac{N}{2}+1}J_{\frac{N}{2}-1}(s)$ for some constant $c>0$, where $J_{\frac{N}{2}-1}(s)$ represents the Bessel function of the first kind of order $\frac{N}{2}-1$.
It is known that $J_{\frac{N}{2}-1}(s)$ has the first positive zero which we denote by $j_{\frac{N}{2}-1}$.
Since $\tpsi_*$ satisfies the linear equation, the zero $j_{\frac{N}{2}-1}$ is simple.
Hence, when $\lambda$ is large, $\tpsi(s)$ also has the first positive zero, which we denote by $s_1(\lambda)$.
By the uniform convergence of $\tpsi(s)$ to $\tpsi_*(s)$ and the simplicity of $j_{\frac{N}{2}-1}$ we see that $s_1(\lambda)\to j_{\frac{N}{2}-1}$ $(\lambda\to\infty)$.
The first positive zero $r_1(\lambda)$ of $\psi(\,\cdot\,)$ satisfies that $r_1(\lambda)=\frac{s_1(\lambda)}{2\sqrt{\lambda}}$.
Therefore,
$\lim_{\lambda\to\infty}r_1(\lambda)=\lim_{\lambda\to\infty}\frac{s_1(\lambda)}{2\sqrt{\lambda}}=0$.
This indicates that $\Theta(\lambda)\to 0$ as $\lambda\to\infty$.
Hence
\begin{equation}\label{S8L1E2}
\lambda_1(\Theta)\to\infty\ \textrm{as}\ \Theta\downarrow 0.
\end{equation}
Because of (\ref{S8L1E1}) and (\ref{S8L1E2}), $\lambda_1(\Theta)$ is defined on $(0,\pi)$.
The proof is complete.
\end{proof}
We study the case where $p>1$ is close to $1$.
Let $\Theta_0\in(0,\pi)$ be fixed.
Since $1<p<p_{\rm S}$, Proposition~\ref{CriSub} says that there is a unique $\Gamma>0$ such that (\ref{EFODE}) with $U(0)=\Gamma$ has a solution for $\Theta=\Theta_0$.
Since $\Gamma$ depends on $p$, we denote $\Gamma$ by $\Gamma(p)$.

We follow the idea of Yanagida~\cite[Theorem~2.6]{YY96} to prove Theorem~\ref{ThE}.
Now we fix $\lambda_1>0$.
Then, by Lemma~\ref{S8L1}, there exists a unique $\Theta_1\in(0,\pi)$ such that (\ref{S8E1}) with $(\lambda,\Theta)=(\lambda_1,\Theta_1)$ has a positive solution.

We set the following problem
\begin{equation}\label{S8E2}
\begin{cases}
W''+(N-1)\frac{\cos\theta}{\sin\theta}W'+\lambda_1W^p=0, & 0<\theta<\Theta_1,\\
W(\Theta_1)=0,\\
W(\theta)>0, & 0<\theta<\Theta_1,\\
W'(0)=0.
\end{cases}
\end{equation}
Since Proposition~\ref{CriSub} is valid for (\ref{S8E2}), (\ref{S8E2}) has a unique solution $W(\theta,p)$ provided that $1<p<p_{\rm S}$.
Let $\Gamma_1(p):=W(0,p)$.
We also consider the initial value problem
\begin{equation}\label{S8E3}
\begin{cases}
Z''+(N-1)\frac{\cos\theta}{\sin\theta}Z'+\lambda_1|Z|^{p-1}Z=0, & 0<\theta<\pi,\\
Z(0)=\Gamma,\ Z'(0)=0.\\
\end{cases}
\end{equation}
Then the following holds:
\begin{lemma}\label{S8L2}
There exists a unique $\Gamma^{\dagger}>0$ such that $\Gamma_1(p)\to\Gamma^{\dagger}$ as $p\downarrow 1$.
\end{lemma}
\begin{proof}
Let $\phi(\theta)$ be a solution of (\ref{S8E1}) with $(\lambda,\Theta)=(\lambda_1,\Theta_1)$.
Then, as $p\downarrow 1$, the solution $Z(\theta)$ of (\ref{S8E3}) converges to $\Gamma\phi(\theta)$ uniformly on $[0,\Theta_1]$.
Applying Green's formula for $Z$ and $\Gamma\phi$, we obtain
\begin{equation}\label{S8L2E1}
(Z'(\theta)\phi(\theta)-Z(\theta)\phi'(\theta))\sin^{N-1}\theta=-\lambda_1(p-1)F(\theta,\Gamma,p),
\end{equation}
where
\[
F(\theta,\Gamma,p):=\int_0^{\theta}\frac{|Z(\varphi)|^{p-1}-1}{p-1}Z(\varphi)\phi(\varphi)\sin^{N-1}\varphi d\varphi.
\]
Since $Z$ converges to $\Gamma\phi$ uniformly on $[0,\Theta_1]$,
\[
\lim_{p\downarrow 1}F(\Theta_1,\Gamma,p)=\Gamma\int_0^{\Theta_1}(\log\Gamma+\log\phi)\phi(\varphi)^2\sin^{N-1}\varphi d\varphi.
\]
Hence, there exists a unique $\Gamma^{\dagger}\in\R$ such that $\lim_{p\downarrow 1}F(\Theta_1,\Gamma,p)=0$ if and only if $\Gamma=\Gamma^{\dagger}$.

We prove the lemma by contradiction.
We assume that there exists some $\delta>0$ such that $\Gamma_1(p)\not\in [\Gamma^{\dagger}-\delta,\Gamma^{\dagger}+\delta]$ as $p\downarrow 1$.
Let $\Gamma=\Gamma_1(p)$. Then, $Z(\theta)=W(\theta)$ on $[0,\Theta_1]$.
The left-hand side of (\ref{S8L2E1}) is $0$ at $\theta=\Theta_1$.
On the other hand, the right-hand side of (\ref{S8L2E1}) is some non-zero constant at $\theta=\Theta_1$ when $p$ is close to $1$.
This is a contradiction, and therefore, $\Gamma_1(p)\to\Gamma^{\dagger}$ as $p\downarrow 1$.
\end{proof}

\begin{proof}[Proof of Theorem~\ref{ThE}]
We take the same $\lambda_1$ and $\Theta_1$ as above.
Let $W$ be the solution of (\ref{S8E2}).
Let $U(\theta):=\lambda_1^{\frac{1}{p-1}}W(\theta)$.
Then $U$ is a solution of (\ref{EFODE}) with $\Theta=\Theta_1$ and $\Gamma(p)=U(0)=\lambda_1^{\frac{1}{p-1}}\Gamma_1(p)$.
By Lemma~\ref{S8L2} we see the following:\\
(i) If $\lambda_1>1$, then $\lambda_1^{\frac{1}{p-1}}\Gamma_1(p)\to\infty$ as $p\downarrow 1$.\\
(ii) If $\lambda_1=1$, then $\lambda_1^{\frac{1}{p-1}}\Gamma_1(p)\to\Gamma^{\dagger}$ as $p\downarrow 1$.\\
(iii) If $\lambda_1<1$, then $\lambda_1^{\frac{1}{p-1}}\Gamma_1(p)\to 0$ as $p\downarrow 1$.\\
Here, by Lemma~\ref{S8L1}, there exists some $\Theta^{\dagger}\in(0,\pi)$ such that $\Theta_1>\Theta^{\dagger}$ for $\lambda_1<1$, $\Theta_1=\Theta^{\dagger}$ for $\lambda_1=1$, and $\Theta_1<\Theta^{\dagger}$ for $\lambda_1>1$.
Thus, the statement of Theorem~\ref{ThE} holds.
\end{proof}



\end{document}